\documentclass[graybox,envcountchap,sectrefs]{svmono}

\usepackage{helvet}
\usepackage{courier}
\usepackage{type1cm}         

\usepackage{enumerate}
\usepackage{makeidx}         
\usepackage{graphicx}        
\usepackage{latexsym,amssymb,amsfonts,amsmath,dsfont}
\usepackage{fancyhdr}
\usepackage{epsfig}
\usepackage{epic}
\usepackage{eepic}
\usepackage{amsmath}
\usepackage{amssymb}
\usepackage{threeparttable}
\usepackage{amscd}
\usepackage{here}
\usepackage{graphicx}
\usepackage{lscape}
\usepackage{tabularx}
\usepackage{subfigure}
\usepackage{longtable}
\usepackage{amsfonts}
\usepackage[spanish]{layout}
\usepackage{mathrsfs}
\usepackage[all,cmtip]{xy}         
\usepackage{xcolor}

\graphicspath{{Figures/}}
\usepackage{multicol}        
\usepackage[bottom]{footmisc}

\DeclareMathAlphabet{\mathpzc}{OT1}{pzc}{m}{it}

\newtheorem{deff}{Definition}[chapter]
\newtheorem{theo}[deff]{Theorem}
\newtheorem{prop}[deff]{Proposition}
\newtheorem{lem}[deff]{Lemma}
\newtheorem{cor}[deff]{Corollary}
\newtheorem{exam}[deff]{Example}
\newtheorem{ques}[deff]{Question}
\newtheorem{rem}[deff]{Remark}

\def\E{\mathbb{E}}

\def\L{\mathscr{L}}

\makeindex             

\begin{document}

\author{C.A. Morales}
\title{Recent Topics on Linear Dynamics}

\date{2025}
\maketitle

\frontmatter

\tableofcontents

\begin{preface}

\noindent
Recent advancements in this field have paved the way for an exploration of intricate phenomena that characterize the behavior of dynamical systems. This book, delves into key concepts and emerging trends in the study of linear dynamical systems, offering a comprehensive overview of some of the most significant developments in this dynamic and evolving field.

In Chapter 1, we explore some basic facts from the Banach contracting
principle to the equivalence between hyperbolicity, shadowing and expansivity for invertible operators on Banach spaces.

In Chapter 2, we embark on the fundamental principles that underlie the understanding of hyperbolicity, expansivity, and shadowing phenomena within linear dynamical systems. This chapter serves as a foundational exploration of crucial concepts, providing readers with the necessary groundwork to delve deeper into subsequent chapters.

Chapter 3, extends the discussion beyond the traditional realms of hyperbolicity, exploring the nuances and implications of generalized hyperbolic obehavior in linear dynamical systems \cite{cgp}. This chapter illuminates the broader spectrum of dynamics that arise when considering generalized hyperbolic structures, opening new avenues for research and discovery.

In Chapter 4, the focus shifts to a critical aspect of dynamical systems—their stability under perturbations. Structural stability is a key theme, and the chapter provides insights into the conditions that ensure the persistence of qualitative properties in the face of small perturbations, a crucial consideration in the study of linear dynamics. We present a proof of the Bernardes-Messaoudi Stability theorem about the structural stability of generalized hyperbolic operators \cite{bm1}.

The exploration continues in Chapter 5, where the attention turns to the intricate dynamics surrounding homoclinic points. This chapter delves into the rich behavior exhibited by systems near homoclinic points, shedding light on their significance and implications in the broader context of linear dynamics.

Chapter 6, presents a captivating departure from conventional topics. The exploration of a peculiar operator and its applications widens the scope of linear dynamics, revealing unexpected connections and potential applications in diverse scientific and engineering fields.

Throughout this book, the aim is to provide a cohesive and accessible resource for researchers, academics, and students interested in the cutting-edge developments within linear dynamics. Each chapter is designed to contribute to a deeper understanding of the intricate dynamics exhibited by linear systems and to inspire further exploration into the fascinating realm of linear dynamics. As we embark on this journey, it is our hope that this monograph serves as both a guide and a catalyst for continued advancements in this dynamic and vibrant field of study.

The author penned these notes while visiting Chungnam National University (CNU) in Daejeon, South Korea, in 2022 and the Vietnam Institute for Advanced Study in Mathematics (VIASM) from December 2023 to January 2024.
They serve as supplementary material for lectures on recent topics in linear dynamics by the author at the University of Da Nang in Da Nang, Vietnam, in January 2024. The author extends gratitude to CNU and VIASM for their warm hospitality and generous financial support. Additionally, appreciation is expressed to the University of Da Nang for providing the opportunity to deliver the lectures from which this textbook arose.

\vspace{100pt}

\noindent
VIASM, CNU \& The University of Da Nang, January 2024 \hfill  C.A.M.

\end{preface}

\mainmatter

\chapter{Basics}

\noindent

\noindent
Let $X$ be a metric space. We say that $X$ is separable if it has a dense sequence.
A sequence $(x_n)_{n\in\mathbb{N}}$ is {\em Cauchy} if
$$
\lim_{m,n\to\infty}d(x_n,x_m)=0.
$$
More precisely, for every $\epsilon>0$ there is $N\in\mathbb{N}$ such that
$d(x_n,x_m)<\epsilon$ for all $m,n\geq N$.
We say that $X$ is complete if every Cauchy sequence of $X$ is convergent i.e. there is $p\in X$ such that
$$
\lim_{n\infty}x_n=p.
$$
\begin{definition}
A map $f:X\to X$ is a {\em contraction} if there is $0<\lambda<1$ such that
\begin{equation}
\label{votar}
d(f(x),f(y))\leq \lambda d(x,y),\quad\quad\forall x,y\in X.
\end{equation}
\end{definition}
A {\em fixed point} of $f:X\to X$ is a point $x\in X$ such that $f(x)=x$.
The set of fixed points of $f$ is denoted by $Fix(f)$.

\begin{theo}[Banach Contracting Principle]
Every contracting map of a complete metric space has a fixed point.
\end{theo}

\begin{proof}
Let $0<\lambda<1$ be satisfying \eqref{votar}.
Fix $x\in X$.

For every $n\geq0$ we define $f^0=id_X$ (the identity of $X$) and
$f^n=f\circ\overset{(n)}{\cdots}\circ f$ for $n\in\mathbb{N}$.
It follows that
$$
d(f^n(x),f^n(y))\leq \lambda d(f^{n-1}(x),f^{n-1}(x))\leq\overset{(n)}{\cdots} \leq \lambda^nd(x,y),
\quad\forall x,y\in X,\, n\in\mathbb{N}.
$$
Now, fix $x\in X$.
Since for all $m>n$ one has
\begin{eqnarray*}
d(f^m(x),f^n(x))&\leq& d(f^m(x),f^{n+1}(x))+d(f^{n+1}(x),f^n(x))\\
&\leq& d(f^m(x),f^{n+2}(x))+d(f^{n+2}(x),f^{n+1}(x))+\\
& & d(f^{n+1}(x),f^n(x))\\
&\vdots& \\
&\leq& \sum_{k=0}^{m-n-1}d(f^{n+k+1}(x),f^{n+k}(x))\\
&\leq&\left( \sum_{k=0}^{m-n-1}\lambda^{n+k}\right)d(f(x),x)\\
&<&\lambda^n\left( \sum_{k=0}^\infty\lambda^{k}\right)d(f(x),x)\\
&=& \frac{\lambda^n}{1-\lambda}d(f(x),x)\to0\quad\quad\mbox{ as }\quad\quad n\to\infty
\end{eqnarray*}
thus
$$
\lim_{n,m\to\infty}d(f^m(x),f^n(x))=0
$$
hence $(f^n(x))_{n\geq0}$ is Cauchy. Since $X$ is complete, there is $p\in X$ such that
$$
\lim_{n\to\infty}f^n(x)=p.
$$
Since $f$ is contracting, $f$ is continuous.
So,
$$
f(p)=f(\lim_{n\to\infty}f^n(x))=\lim_{n\to\infty}f^{n+1}(x)=p
$$
thus $p\in Fix(f)$ completing the proof.
\qed
\end{proof}

Note that if $f$ is a contraction of a complete metric space, then the above proof says that there is $p\in Fix(f)$ such that
$$
\lim_{n\to \infty}f^n(x)=p,\quad\quad\forall x\in X.
$$
In other words, the positive orbit $(f^n(x))_{n\geq0}$ of every point $x\in X$
converges to $p$.
Therefore, the point $p$ is an {\em attractor} in the sense that it
{\em attracts} the {\em positive orbit} of every point.

Now, we introduce some dynamical formalism.
Let $f:X\to X$ be a map of a metric space.
We say that $x\in X$ is
\begin{itemize}
\item
Fixed if $f(x)=x$;
\item
Periodic if $f^k(x)=x$ for some $k\in\mathbb{N}$;
\item
{\em Nonwandering} if for every neighborhood $U$ of $x$ there is $n\in\mathbb{N}$ such that $f^n(U)\cap U\neq\emptyset$;
\item
{\em Chain recurrent} if for every $\delta>0$ there is a sequence
$(x_i)_{i=0}^k$ such that $x_0=x=x_k$ and $d(f(x_i),x_{i+1})<\delta$ for all $0\leq i\leq k-1$.
\end{itemize}

The corresponding sets are denoted by $Fix(f)$, $Per(f)$, $\Omega(f)$ and $CR(f)$.
The following relationships between these sets holds:
$$
Fix(f)\subset Per(f)=\bigcup_{n\in\mathbb{N}}Fix(f^n)\subset \Omega(f)\subset CR(f).
$$
None of the reversed of the above inclusions holds.
We also have for any $x\in X$,

\begin{itemize}
\item
Positive orbit: $(f^n(x))_{n\geq0}$;
\item
Negative orbit (when $f$ invertible): $(f^{-n}(x))_{n\geq0}$ where $f^{-n}=(f^{-1})^n$;
\item
Orbit (when $f$ invertible): $(f^n(x))_{n\in\mathbb{Z}}$.
\end{itemize}

Of fundamental importance is the knowledge of the limit points of the orbits of $x$. In this vein, we define
$\omega(x)$ as the limit points of the positive orbit of $x$ namely
$$
\omega(x)=\{y\in X:y=\lim_{i\to\infty}f^{n_i}(x)\mbox{ for some sequence }n_i\to\infty\}.
$$
In the invertible case we define $\alpha(x)$ as the $\omega$-limit set of $f$ under the negative orbit.

The Banach principle can be reformulated as follows:

If $fX\to X$ is a contraction of a complete metric space, then there is $p\in X$ such that
$Fix(f)=\Omega(f)=\{p\}$ and $\omega(x)=\{p\}$ for all $x\in X$.

In other words, we know what will happen with any orbit:
They will be attracted to a fixed point.
This kind of result is the prototype of what we search in dynamics.
More precisely, to understand the asymptotic behavior of the orbits or, more precisely, to determinate the structure of the omega-limit sets, or, the nonwandering or chain-recurrent points.

One more structure in dynamics is as follows.

\begin{itemize}
\item
We say that $x\in X$ is a homoclinic point of a map $f:X\to X$, associated to $p\in X$, if
$$
\lim_{n\to\pm\infty}f^n(x)=p.
$$
\end{itemize}
The set of homoclinic points associated to $p$ is denoted by $H_f(p)$.
Clearly $p\in Fix(f)$ $\iff$ $H_f(p)\neq\emptyset$.

\begin{definition}
We say that $f:X\to X$ is:
\begin{itemize}
\item
{\em Expansive}: There is $\epsilon>0$ (called expansivity constant) such that if
$d(f^n(x)-f^n(y)\leq\epsilon$ for all $n\in\mathbb{Z}$ $\Longrightarrow$ $x=y$.
\item
We say that $f:X\to X$ is {\em shadowing} (or that has the {\em shadowing property}) if
for every $\epsilon>0$ there is $\delta>0$ such that for every sequence
$(x_n)_{n\in\mathbb{Z}}$ with $d(f(x_n),x_{n+1})\leq\delta$ for all $n\in\mathbb{Z}$ (called {\em $\delta$-pseudo orbit}) there is
$x\in X$ such that $d(f^n(x),x_n)\leq\epsilon$ for all $n\in\mathbb{Z}$.
\end{itemize}
\end{definition}

Next, we make some comments about the origin and motivations of the Linear Dynamics Theory.

We say that $X$ is a {\em Banach space} over $K=\mathbb{R}$ or $\mathbb{C}$ if it is a vector space over $K$ equipped with a norm $\|\cdot\|$ whose induced metric is complete.

\begin{example}
$X=\mathbb{C}^n$ with the norm $\|z\|=\sqrt{|z_1|^2+\cdots+|z_n|^2}$ or, for a given $p\in [0,\infty)$
$X=l^p(\mathbb{C})=\{z=(z_i)_{i\geq0}:z_i\in\mathbb{C}$ and $\sum_{i\geq0}|z_i|^p<\infty\}$ with the norm $\|z\|=(\sum_{i\geq0}|z_i|^p)^{\frac{1}p}$.
Similarly $X=\mathbb{R}^n$ and $l^p(\mathbb{R})$.
\end{example}

\begin{example}
Let $(X,\mu)$ be a measure space. Given $p\in [1,\infty)$ we define
$$
L^p=L^p(X,\mu)=\{f:X\to \mathbb{R}:f\mbox{ is measurable and }\int|f|^pd\mu<\infty\}.
$$
This is a Banach space (with the usual sum and scalar product of functions) and the $L^p$-norm
$$
\|f\|_p=\left(\int|f|^pd\mu\right)^{\frac{1}p}.
$$
\end{example}

A Banach space is {\em Hilbert} if there is an inner product $\langle\cdot,\cdot\rangle$ such that
$$
\|x\|^2=\langle x,x\rangle,\quad\quad\forall x\in X.
$$
Of the above examples, only $\mathbb{C}^n$, $l^2$ and $L^2$ are Hilbert.

\begin{definition}
A map of a Banach space $L:X\to X$ is linear if $L(x+y)=L(x)+L(y)$ and $L(\lambda x)=\lambda L(x)$ for all $x,y\in X$ and $\lambda\in\mathbb{C}$.
\end{definition}

\begin{definition}
A linear operator of a Banach space is a linear map which is also continuous.
\end{definition}

Following properties are equivalent for every linear map of a Banach space
$L:X\to X$,

\begin{itemize}
\item
$L$ is a linear operator (i.e. continuous);
\item
$L$ is continuous at $x=0$.
\item
$L$ is {\em bounded} i.e.
there is $K\in (0,\infty)$ such that
$$
\|L(x)\|\leq K\|x\|,\quad\quad\forall x\in X.
$$
\item
$\sup_{x\neq0}\frac{\|L(x)\|}{\|x\|}<\infty$
\end{itemize}

We can then define the {\em norm} of a bounded linear operator $L:X\to X$ by
$$
\|L\|=\sup_{x\neq0}\frac{\|L(x)\|}{\|x\|}
$$
and one has
$$
\|L\|=\sup_{\|x\|=1}\|L(x)\|.
$$
The space $Op(X)$ of linear operators of a Banach space $X$ is a Banach space if endowed with the norm $\|L\|$ for $L\in Op(X)$.
We also consider those $L\in Op(X)$ such that $L^{-1}$ exists and is continuous. This is the set of {\em invertible linear operators} denoted by
$GL(X)$.

The {\em spectrum} of $L\in Op(X)$ is the set $\sigma(L)\subset \mathbb{C}$ defined by $\lambda\in \sigma(L)$ $\iff$ $\lambda-L\notin GL(X)$.
As a result, $\sigma(L)$ is precisely the set of $\lambda\in \mathbb{C}$ such that the linear operator $\lambda-L$ is not bijective i.e. not injective or not surjective. In the first case $\lambda$ is called {\em eigenvalue} of $L$.
The set of eigenvalues is denoted by $\sigma_p(L)$ (this is the so-called the point spectrum of $L$).

The origin of the Linear Dynamics can be traced back to the {\em Invariant subspace problem} (ISP):

\begin{prop}
Does every linear operator of a complex Banach space $L:X\to X$
exhibit a closed subspace $M$ different from $X$ and $\{0\}$ such that
$L(M)\subset M$?
\end{prop}

The answer is immediately positive if $X$ is not separable.

In fact, by taking any $x\in X\setminus \{0\}$ the closure $M$ of subspace generated by the positive orbit of $x$ under $L$ namely
\begin{equation}
\label{midway}
M=\overline{Span\{L^n(x):n\geq0\}}
\end{equation}
satisfies the requirements of the ISP.

The question therefore makes sense in the separable case only.

In such a case a negative answer was given by Enflo.
However, the ISP is still open on separable Hilbert spaces.
The identity \eqref{midway} put in evidence the relationship between the dynamics of a linear operator (represented by the positive orbit) and the ISP.
Many authors have been studying the dynamics of linear operators on Banach spaces.

This is what the subject Linear Dynamics is about.
We first analyze conditions for chain recurrence, homoclinic points and topological transitivity for linear operators.

\begin{theo}
If $L\in Op(X)$ for a Banach space $X$, then $CR(L)$ is a closed subspace
and $L(CR(L))\subset CR(L)$.
\end{theo}

\begin{proof}
Suppose that $x\in CR(L)$. If $\lambda\neq0$ is an scalar and $\epsilon>0$, there is a sequence $(x_i)_{i=0}^k$ such that $x_0=x_k=x$ and $\|L(x_i)-x_{i+1}\|\leq \frac{\epsilon}{|\lambda|}$ for all $0\leq i\leq k-1$. It follows that $(\lambda x_i)_{i=0}^k$ is a sequence
with $\lambda x_0=\lambda x_k=\lambda x$ and $\| L(\lambda x_i)-\lambda x_{i+1}\|\leq\epsilon$ for all $0\leq i\leq k$. Hence $\lambda x\in CR(L)$ for $\lambda\neq0$.
Since $0=0 x$ always belong to $CR(L)$ we obtain $Spam(x)\subset CR(L)$.

If $x,y\in CR(L)$ there are sequence, $(x_i)_{i=0}^k$ and $(y_i)_{i=0}^l$ with
$x_0=x=x_k$, $y_0=y=y_l$, $\|L(x_i)-x_{i+1}\|\leq\frac{\epsilon}2$ for $0\leq i\leq k-1$
and $\|L(y_j)-y_{j+1}\|\leq\frac{\epsilon}2$ for $0\leq j\leq l-1$. Assuming that $k\leq l$
the sequence
$$
(z_i)_{i=0}^l=\{x_0+y_0,x_1+y_1,\cdots, x_k+y_k,0+y_{k+1},\cdots, 0+y_l\}
$$
start and finish in $x+y$ and $\|L(z_i)-z_{i+1}\|\leq \epsilon$ for all $0\leq i\leq l-1$.
Therefore, $x+y\in CR(L)$.
\end{proof}

\begin{theo}
\label{guadal}
If $L\in GL(X)$ for a finite dimensional Banach space,
then $L$ has no nonzero homoclinic points associated to $0$.
\end{theo}

\begin{proof}
We have denoted by $H_L(0)$ the set of homoclinic points associated to $0$.
We can prove easily that $H_L(0)$ is a subspace hence closed (for $dim(X)<\infty$).
It is also invariant namelt $L(H_L(0))=H_L(0)$ thus we can assume $H_L(0)=X$.
Let $\lambda\in \sigma(L)$ hence $\lambda\in \sigma_p(L)$ for $dim(X)<\infty$.
Take $x_0\in X$ with $\|x_0\|=1$ such that $L(x_0)=\lambda x_0$.
If $|\lambda|>1$ (resp. $|\lambda|<1$) then
$\|L^n(x_0)\|=|\lambda|^n\to\infty$ as $n\to\infty$ (resp. $n\to-\infty$) which contradicts
$x_0\in H_L(0)=X$.
If $|\lambda|=1$, $\|L^n(x_0)\|=|\lambda|^n=1$ for all $n\in\mathbb{Z}$ contradicting one
more that $x_0\in H_L(0)=X$.
These contradictions complete the proof.
\qed
\end{proof}

\begin{definition}
A map of a metric space $f:X\to X$ is {\em topologically transitive}
if for all open sets $U,V\subset X$ there is $n\in\mathbb{N}$ such that
$f^n(U)\cap V\neq\emptyset$.
Topologically transitive linear operators are often called {\em hypercyclic operators}.
\end{definition}

Since Banach spaces are complete, to be topologically transitive is equivalent to the existence of dense orbits.
The natural question is if there are hypercyclic operators on finite dimensional spaces. However,

\begin{theo}
\label{paya}
There are no hypercyclic operators on finite dimensional Banach spaces.
\end{theo}

Since every linear operator of a finite-dimensional Banach space has eigenvalues, this theorem is a direct consequence of the following
lemma which is interesting by its own. We learn it from the nice book by Bayart and Matherson \cite{bm}.
Remember that the dual of a linear operator of a Banach space $L:X\to X$ is the linear operator $L^*(\varphi)=\varphi\circ L$ defined on the dual space $X^*$ i.e. the space of continuous linear mappings $\varphi:X\to \mathbb{C}$.

\begin{lem}
If $L\in Op(X)$ is hypercyclic, then its dual operator $L^*$ has no eigenvalues.
\end{lem}

\begin{proof}
Otherwise, there is $\lambda\in \mathbb{C}$ such that
$L^*(\varphi)=\lambda\varphi$ for some continuous non-zero linear mapping $\varphi:X\to \mathbb{C}$.
Since $L$ is hypercyclic, there is $x\in X$ such that
$(L^n(x)_{n\in\geq0}$ is dense in $X$.
Since $L^*(\varphi)=\lambda\varphi$, we have $\varphi\circ L=\lambda \varphi$ so $\varphi(L^n(x))=\lambda^n\varphi(x)$ for every $n\geq0$.
Since $\varphi$ is continuous and $(L^n(x))_{n\geq0}$ is dense in $X$, we have that $(\lambda^n\varphi(x))_{n\geq0}$ is dense in $\varphi(X)$.
But $\varphi$ is non-zero so $\varphi$ is onto i.e. $\mathbb{C}=\varphi(X)$ thus $(\lambda^n\varphi(x))_{n\geq0}$ is dense in $\mathcal{C}$ which is absurd. This completes the proof.
\end{proof}

To find examples in infinite dimension we use:

\begin{theo}[The Hypercyclicity Criterium]
Let $L\in Op(X)$ for a Banach space $X$.
Suppose that there are two dense subsets $D,D'\subset X$ as well as
sequences of positive integers $n_k\to\infty$ and maps $S_{n_k}:D'\to X$ such that
\begin{itemize}
\item
$L^{n_k}(x)\to 0$ for every $x\in D$;
\item
$S_{n_k}(y)\to 0$ for every $y\in D'$;
\item
$L^{n_k}S_{n_k}(y)\to y$ for every $y\in D'$.
\end{itemize}
Then, $L$ is hypercyclic.
\end{theo}

\begin{proof}
Take open sets $U$ and $V$ in $X$.
There are $x\in U\cap D$ and $y\in V\cap D'$.
We have $x+S_{n_k}(y)\in V$ and $T^{n_k}(x+S_{n_k}(y))\to V$ for $k$ large.
Then, $n=n_k$ and $z=x+S_{n_k}(y)$ satisfy $z\in U\cap L^{-n}(V)$
thus $L^n(U)\cap V\neq\emptyset$ completing the proof.
\end{proof}

\begin{remark}
All known hypercyclic operators satisfy HC.
The question if there are hypercyclic operators not satisfying HC is still open.
\end{remark}

We apply this criterium to the following result.

\begin{theo}[Rolewicz Example]
There is a hypercyclic operator.
\end{theo}

\begin{proof}
Let $X=l^2$ and $L\in Op(X)$ be defined by
$$
L(\xi_0,\xi_1,\xi_2,\cdots)=(2\xi_1,2\xi_2,\cdots),\quad\quad\forall \xi=(\xi_1,\xi_2,\cdots)\in l^2.
$$
Then, apply the HC to
$$
D=\{\xi\in l^2:\xi_i=0\mbox{ except by finitely many }i\geq0\},
$$
$$
D'=X,\quad n_k=k\quad\mbox{ and }\quad
S_{n_k}:D'\to X
$$
defined by
$$
S_{n_k}(\xi)=(0,\overset{(k)}{\cdots}, 0,\frac{1}{2^k}\xi_0,\frac{1}{2^k}\xi_1,\cdots).
$$
\qed
\end{proof}

Next we define hyperbolic operators and discuss some equivalences.

\begin{definition}
\label{lockdown'}
We say that $L\in GL(X)$ is {\em hyperbolic} if $\sigma(L)\cap S^1=\emptyset$.
\end{definition}

It follows from the Riesz spectral decomposition theorem that $L$ is hyperbolic $\iff$ there is a direct sum decomposition
$X=S\oplus U$ by closed subspaces $S$ and $U$ such that
$$
L(S)=S,\quad L(U)=U,\quad r(L|_S)<1\quad\mbox{and}\quad r(L^{-1}|_U)<1
$$
where $r(P)=\sup\{|\lambda|:\lambda\in \sigma(P)\}$ is the spectral radius of operator $P$.

By Gelfand's spectral radius formula this is equivalent
to the existence of a direct sum $X=S\oplus U$ by closed subspace $S,U$ and positive constant $\lambda,K$ such that
\begin{itemize}
\item
$L(S)=S$ and $L(U)=U$.
\item
$\|L^n(x)\|\leq Ke^{-\lambda n}\|x\|$ for all $x\in S$ and $n\geq0$.
\item
$\|L^n(x)\|\geq K^{-1}e^{\lambda n}\|x\|$ for all $x\in S$ and $n\geq0$.
\end{itemize}

\begin{theo}
\label{nature}
Every hyperbolic $L\in GL(X)$ is expansive and shadowing.
\end{theo}

We split the proof in some lemmas.

\begin{lem}
\label{lili}
Every hyperbolic $L\in GL(X)$ is expansive.
\end{lem}

\begin{proof}
Indeed, suppose $L$ is not expansive.
Then, there is some $x\neq0$ such that
the sequence $(\|L^n(x)\|)_{n\in\mathbb{Z}}$ is bounded.
We write $x=x^S+x^U\in S\oplus U$.
Since
$\|L^n(x^U)\|\leq \|L^n(x)\|+\|L^n(x^S)\|$ and $L^n(x^S)\to0$ as $n\to\infty$ (for $x^S\in S$), the sequence $(\|L^n(x^U)\|)_{n\geq0}$ is bounded.
This implies $x^U=0$.
Then, $x=x^S\in S$ so $(L^n(x^S))_{n\leq0}$ is bounded thus $x^S=0$ hence $x=0$ a contradiction. Therefore, $L$ is expansive.
\qed
\end{proof}

To prove that a hyperbolic $L\in GL(X)$ is shadowing is subtle.
We proceed as in Ombach \cite{o}.

\begin{lem}
\label{lll1}
Every contraction of a complete metric space
$f:X\to X$ has the shadowing property.
\end{lem}

\noindent
\begin{proof}
Let $\lambda$ be the contracting constant of $f$.
Given $\epsilon>0$ we choose $\delta=(1-\lambda)\epsilon$.
Let $(x_i)_{i\in\mathbb{Z}}$ be a $\delta$-pseudo orbit.

Define $E=\{\xi\in X^\mathbb{Z}:d(\xi_i,x_i)\leq\epsilon\,\forall i\in\mathbb{Z}\}$.
$E$ is a complete metric space if equipped with the supremum distance
$$
D(\xi,\eta)=\sup_{i\in\mathbb{Z}}d(\xi_i,\eta_i).
$$

Define $\Gamma:E\to E$ by
$$
\Gamma(\xi)_i=f(\xi_{i-1}),\quad\quad\forall i\in\mathbb{Z}.
$$
We must prove $\Gamma$ is well-defined.
Since
$$
d(f(\xi_{i-1}),x_i)\leq d(f(\xi_{i-1}),f(x_{i-1}))+d(f(x_{i-1}),x_i)\leq \lambda\epsilon+\delta=\epsilon,
$$
we have $\Gamma(E)\subset E$.

Since
$$
D(\Gamma(\xi),\Gamma(\eta))=\sup_{i\in\mathbb{Z}}d(f(\xi_i),f(\eta_i))\leq \lambda
D(xi,\eta),
$$
for all $\xi,\eta\in E$, one get that $\Gamma$ is a contraction.
Then, $\Gamma$ has a fixed point $\xi$ by the Banach principle.
This implies
$$
f(\xi_{i-1})=\xi_i\quad\mbox{ and }\quad
d(\xi_i,x_i)\leq\epsilon,\quad\forall i\in\mathbb{Z}.
$$
The first identity above implies $\xi_i=f^i(\xi_0)$ for all $i\in\mathbb{Z}$.
Replacing in the second we get
$$
d(f^i(x),x_i)\leq\epsilon,\quad\quad\forall i\in\mathbb{Z}
$$
with $x=\xi_0$ completing the proof.
\qed
\end{proof}

\begin{lem}
\label{lll2}
Let $f:X\to X$ be a uniformly homeomorphism of a metric space.
If $f$ is shadowing, then so does $f^{-1}$.
\end{lem}

\begin{proof}
Fix $\epsilon>0$ and $\delta'$ from the shadowing of $f$.
Since $f$ is a uniform homeomorphisms, there is $\delta>0$ such that
$$
d(a,b)\leq\delta\quad\Longrightarrow\quad d(f(a),f(b))\leq\delta'.
$$
Let $(x_i)_{i\in\mathbb{Z}}$ be a $\delta$-pseudo orbit of $f^{-1}$
i.e.
$d(f^{-1}(x_i),x_{i+1})\leq\delta$ for all $i\in\mathbb{Z}$.
Then, $d(f(x_{i+1}),x_i)\leq\delta'$ for all $i\in\mathbb{Z}$.
Defining $y_i=x_{-i}$ for $i\in\mathbb{Z}$ we get
$d(f(y_i),y_{i+1})\leq\delta$ for all $i\in\mathbb{Z}$.
Then, there is $z\in X$ such that
$d(f^i(z),y_i)\leq\epsilon$ for all $i\in\mathbb{Z}$.
So, $d(f^i(z),x_{-i})\leq\epsilon$ for all $i\in\mathbb{Z}$.
Replacing $i$ by $-i$ we get
$d((f^{-1})^i(z),x_i)\leq\epsilon$ for all $i\in\mathbb{Z}$.
Thus, $f^{-1}$ has the shadowing property completing the proof.
\qed
\end{proof}

\begin{lem}
\label{lll3}
Let $L\in GL(X)$ for a Banach space $X$.
Suppose that there is a direct sum $X=E\oplus F$ by closed subspaces $E$ and $F$ such that
$L(E)=E$ and $L(F)=F$.
Then, restricted operators $L|_E$ and $L|_F$ are shadowing $\iff$ $L$ is shadowing.
\end{lem}

\begin{proof}
($\Longrightarrow$).
Every $x\in X$ writes uniquely as
$x=x^E\oplus x^F\in E\oplus F$.
Since the sum $X=E\oplus F$ is direct, there is a positive constant $\beta$ such that
$\|x^S\|\leq \beta\|x\|$ and $\|x^U\|\leq\beta\|x\|$ for all $x\in X$.

Now, let $\epsilon>0$ and $\delta>0$ be given from the shadowing of both $L|_E$ and $L|_F$
for $\frac{\epsilon}{2}$.
Let $(x_i)_{i\in\mathbb{Z}}$ be a $\frac{\delta}{\beta}$-pseudo orbit of $L$.
Since $L(E)=E$, $L(x^E_i)-x^E_{i+1}=(L(x_i)-x_{i+1})^E$ so
$$
\|L(x^E_i)-x^E_{i+1}\|\leq \beta\|L(x_i)-x_{i+1}\|\leq\beta\frac{\delta}{\beta}= \delta,\quad\quad\forall i\in\mathbb{Z},
$$
thus there is $x^E\in E$ such that
$\|L^i(x^E)-x^E_i\|\leq\frac{\epsilon}2$ for all $i\in\mathbb{Z}$.
Likewise, there is $x^F\in F$ such that
$\|L^i(x^F)-x^F_i\|\leq\frac{\epsilon}2$ for all $i\in\mathbb{Z}$.

Then, $x=x^E\oplus x^F$ satisfies
$$
\|L(x)-x_i\|\leq\|L^i(x^E)-x^E_i\|+\|L^i(x^F)-x^F_i\|\leq\frac{\epsilon}2+\frac{\epsilon}2=\epsilon,\quad\quad\forall i\in\mathbb{Z}
$$
completing the proof.

To prove ($\Longleftarrow$)
denote by $L^E=L|_E$ and $L^F=L|_F$ the restrictions of $L$ to $E$ and $F$ respectively.
We only need to prove that $L^E$ has the shadowing property.
By passing to projections if necessary we can assume
\begin{equation}
\label{anis}
\|x\|=\max\{\|x^E\|,\|x^F\|\}\quad\mbox{ whenever }\quad x=x^E\oplus x^F\in M\oplus N.
\end{equation}
Let $\epsilon>0$ and $\delta>0$ be given by the shadowing property of $L$.
Let $(x^E_i)_{i\in\mathbb{Z}}$ be a $\delta$-PO of $L^E$.
Then, there is $x\in X$ such that
\begin{equation}
\label{daryl}
\|L^i(x)-x^E_i\|\leq \epsilon\quad\quad i\in\mathbb{Z}.
\end{equation}
Writing $x=x^E+x^F$ one has
$$
\|(L^E)^i(x^E)-x^E_i\|\overset{\eqref{anis}} {\leq}\max\{\|L^i(x^E)-x^E_i\|,\|L^i(x^F)\|\}=
\|L^i(x)-x^E_i\|\leq\epsilon,
$$
$\forall i\in\mathbb{Z}$
proving that $L^E$ has the shadowing property.
\qed
\end{proof}

\begin{proof}[of Theorem \ref{nature}]
Let $L\in GL(X)$ be hyperbolic.
By Lemma \ref{lili} we have that $L$ is expansive.
It remains to prove that $L$ is shadowing.

We have the direct sum $X=S\oplus U$ given by the hyperbolicity of $L$.
Since $L$ contracts $S$, $L|_S$ is shadowing by Lemma \ref{lll1}.
Likewise $(L|_U)^{-1}$ contracting so shadowing by Lemma \ref{lll1} thus $L|_U$ is shadowing too by Lemma \ref{lll2}.
We conclude that both $L|_S$ and $L|_U$ are shadowing hence $L$ is by Lemma \ref{lll3}.
This completes the proof.
\qed
\end{proof}

\begin{remark}
In the above proof, $L|_S$ is not really a contraction but an eventual contraction. However, either we can repeat the shadowing proof to this case or else we can change the norm to an equivalent one to turn $L|_S$ into a contraction. Similarly for $L^{-1}|_U$.
\end{remark}

The converse of Theorem \ref{nature} is also true by a recent important result by Bernardes and Messaoudi (to be explained in the next chapter).
More precisely, every expansive and shadowing operator
$L\in GL(X)$ is hyperbolic.

We will finish this chapter by explaining a stronger fact in the finite-dimensional case. More precisely, that hyperbolicity, shadowing and expansivity are all equivalent.

To start with, we define
$$
E^c=\{x\in \sup_{n\in\mathbb{Z}}\|L^n(x)\|<\infty\}.
$$

We have the following elementary result.

\begin{lem}
\label{ele}
$L\in GL(X)$ is expansive $\iff$ $E^c=\{0\}$.
\end{lem}

\begin{proof}
If $L$ is expansive, let $\epsilon$ be given in the definition.
If $x\in E^c$, we have $\sup_{n\in\mathbb{Z}}\|L^n(x)\|<M$ for some constant $M$
thus
$$
\|L^n(\frac{\epsilon x}M)\|\leq \epsilon,\quad\quad\forall n\in\mathbb{Z}
$$
thus $\frac{\epsilon x}M=0$ hence $x=0$ proving $E^c(x)=\{0\}$.
Conversely, if $E^c=\{0\}$ and $\|L^n(x)\|<1$ for all $n\in\mathbb{Z}$, then $x\in E^c=\{0\}$ thus $x=0$ proving the expansivity of $L$.
\qed
\end{proof}

We use this remark to prove the following.

\begin{theo}
Let $X$ be a finite-dimensional Banach space.
If $L\in GL(X)$ is expansive, then $L$ is hyperbolic.
\end{theo}

\begin{proof}
If some  $\lambda\in \sigma_p(L)$ has modulus $1$,
then any unitary eigenvector $x\in X$ associated to $\lambda$
satisfies
$$
\sup_{n\in\mathbb{Z}}\|L^n(x)\|=1
$$
hence $x\in E^c$. But since $L$ is expansive, $E^c=\{0\}$ thus $x=0$ contradicting that $x$ is unitary.
Since $dim(X)<\infty$, $\sigma(L)=\sigma_p(L)$, so
we have proved $\sigma(L)\cap S^1=\emptyset$ proving that $L$ is hyperbolic.
\qed
\end{proof}

To state the next lemma we need the following concept.

\begin{definition}
A {\em linear isometry} of a Banach space $X$ is a
linear homeomorphism $L:X\to X$ such that $\|L(x)\|=\|x\|$ for every $x\in X$.
\end{definition}

\begin{lem}
\label{l2-A}
A linear isometry of a nontrivial Banach space does not have the shadowing property.
\end{lem}

\begin{proof}
Suppose that there is a Banach space $X$ and $L\in GL(X)$ with the shadowing property.
Take $\delta>0$ from this property for $\epsilon=1$,
and let $x\in X$ be an arbitrary point.

Choose a sequence
$0=q_0, q_1,\cdots, q_r=x$ such that
$\|q_{i+1}-q_i\|\leq\delta$ for $0\leq i\leq r-1$.

Define $(p_i)_{i\in\mathbb{Z}}$ by
$$
p_i = \left\{ \begin{array}{rcl}
L^i(q_0),& \mbox{if} & i<0\\
& & \\
L^i(q_i), & \mbox{if} & 0\leq i\leq r\\
& & \\
L^i(q_r), & \mbox{if} & r<i.
\end{array} \right.
$$

Since
$$
\|L(p_i)-p_{i+1}\|=\|L^{i+1}(q_i)-L^{i+1}(q_{i+1})\|=\|q_{i+1}-q_i\|\leq\delta,
$$
$\forall 0\leq i\leq r-1,$
one has $\|L(p_i)-p_{i+1}\|\leq\delta$ for $i\in\mathbb{Z}$ thus
there is $z\in X$ such that
$$
\|L^i(z)-p_i\|\leq 1,\quad\quad\forall i\in\mathbb{Z}.
$$
Since $L$ is a linear isometry,
$$
\|z-q_i\|=\|z-L^{-i}(p_i)\|=\|L^i(z)-p_i\|\leq1,\quad\quad\forall 0\leq i\leq r.
$$

In particular,
$\|z\|=\|z-0\|=\|z-q_0\|\leq1$ and $\|z-x\|=\|z-q_r\|\leq1$ so
$$
\|x\|\leq \|z\|+\|z-x\|\leq 2,\quad\quad\forall x\in X
$$
hence $X=\{0\}$ a contradiction. This completes the proof.
\qed
\end{proof}

\begin{lem}
\label{le}
If $L\in GL(X)$ is shadowing for a finite dimensional Banach space $X$, then $L$ is hyperbolic.
\end{lem}

\begin{proof}
Suppose that there is $\lambda\in \sigma(L)\cap S^1$.
Then, $\lambda$ is an eigenvalue and there is a direct sum $X=E\oplus F$ by closed subspaces
$E$ and $F$ where $F$ is the eigenspace associated to $\lambda$.
By Lemma \ref{lll3} we have that $L|_F$ has the shadowing property.
But $L|_F$ is an isometry so we have a contradiction by Lemma \ref{l2-A}.
Hence $\sigma(L)\cap S^1=\emptyset$ so $L$ is hyperbolic.
\qed
\end{proof}

We have then proved the following result:

\begin{theo}
The following properties are equivalent for any invertible linear operator of a finite-dimensional Banach space $L:X\to X$:
\begin{itemize}
\item
$L$ is hyperbolic.
\item
$L$ is expansive.
\item
$L$ has the shadowing property.
\end{itemize}
\end{theo}

\chapter{Hyperbolicity, expansivity and shadowing}

\noindent
Let $X$ be a (complex) Banach space.
Denote by $GL(X)$ the set of linear homeomorphisms $L:X\to X$.
We will work with the following notions of expansivity:

\begin{deff}
We say that $L\in GL(X)$ is
\begin{itemize}
\item
 {\em uniformly expansive} if
there is $m\in\mathbb{N}$ such that
either $\|L^m(x)\|\geq2$ or $\|L^{-m}(x)\|\geq2$ for every $x\in X$ with $\|x\|=1$;
\item
{\em asymptotically expansive} \cite{lmv} if there is $\epsilon>0$ (called asymptotic expansivity constant) such that
if $\|L^n(x)\|\leq\epsilon$ for all $n\geq0$, then $\lim_{n\to\infty}L^n(x)=0$.
\end{itemize}
\end{deff}

It follows that $L$ is hyperbolic if and only if there is a splitting $X=S\oplus U$ formed by closed subspaces $S$ and $UY$ such that
$L(S)=S$, $L^{-1}(U)=U$, $r(L|_S)<1$ and $r(L^{-1}|_U)<1$.

In this chapter we will prove the result below.

\begin{theo}
\label{thA-1}
The following properties are equivalent for every $L\in GL(X)$:
\begin{enumerate}
\item
$L$ is expansive and has the shadowing property.
\item
$L$ is asymptotically expansive and has the shadowing property.
\item
$L$ is hyperbolic.
\end{enumerate}
\end{theo}

Item (3) implies Item (1) (results by Eisenberg-Hedlund and \cite{o}).
That Item (1) implies Item (3) was proved Bernardes and Messaoudi \cite{bm2}.
The equivalences involving (2) are new as far as we known.

To prove this theorem we will use the following notations and lemmas.
Given $L\in GL(X)$ we define
$$
E^{cs}=\{x\in X: \sup_{n\geq0}\|L^n(x)\|<\infty\}
\mbox{ and } E^{cu}=\{x\in X:\sup_{n\geq0}\|L^{-n}(x)\|<\infty\}.
$$
Clearly $E^{cs}$ and $E^{cu}$ are subspaces which are {\em invariant} i.e.
$L(E^{cs})=E^{cs}$ and $L(E^{cu})=E^{cu}$.

Clearly, $E^{cs}\cap E^{cu}=E^c$
so, by Lemma \ref{ele}, $L$ is expansive if and only if $E^{cs}\cap E^{cu}=\{0\}$.
A further property is given below.

\begin{lem}
\label{l1-1}
If $L\in GL(X)$ has the shadowing property, then
$X=E^{cs}+E^{cu}$.
\end{lem}

\begin{proof}
Since both $E^{cs}$ and $E^{cu}$ are subspaces, $E^{cs}+E^{cu}$ also is.
Hence it suffices to find $\delta>0$ such that $B(0,\delta)\subset E^{cs}+E^{cu}$,
where $B(a,r)=\{x\in X:\|x-a\|<\delta\}$ is the open ball.

Take $\delta>0$ from the shadowing property of $L$ for $\epsilon=1$.
Take $x\in B(0,\delta)$ i.e. $\|x\|<\delta$.
Define the sequence $(x_n)_{n\in\mathbb{Z}}$ by
$$
x_n = \left\{ \begin{array}{rcl}
0,& \mbox{if} &n<0\\
& & \\
L^n(x), & \mbox{if} & n\geq0.
\end{array} \right.
$$
Then, $\|L(x_n)-x_{n+1}\|<\delta$ for every $n\in\mathbb{Z}$ so there is $y\in X$ such that
$$
\|L^n(y)-x_n\|\leq 1,\quad\quad\forall n\in\mathbb{Z}.
$$
It follows that $\|L^n(y)\|\leq 1$ for every $n<0$ so $y\in E^{cu}$.
Also $\|L^n(y-x)\|\leq 1$ for $n\geq0$ so $y-x\in E^{cs}$.
Hence
$x=x-y+y\in E^{cs}+E^{cu}$ proving $B(0,\delta)\subset E^{cs}+E^{cu}$.
Since $E^{cs}+E^{cu}$ is a subspace, $X=E^{cs}+E^{cu}$ completing the proof.
\qed
\end{proof}

The expansive alternative of the lemma below was noted in \cite{bcdmp}.

\begin{lem}
\label{l2-1}
If $L\in GL(X)$ is expansive (or asymptotically expansive) and has the shadowing property, then
$L$ is uniformly expansive.
\end{lem}

\begin{proof}
By a theorem of Hedlund \cite{h} it suffices to show that
$\sigma_a(L)\cap S^1=\emptyset$ where $\sigma_a(L)$ is the approximated point spectrum.

Suppose not namely that there is $\lambda\in \sigma_a(L)\cap S^1$.
We choose $\epsilon=\frac{1}2$ and take $\delta$ from the shadowing property of $L$ for this $\epsilon$.
Because $\lambda\in \sigma_a(L)$, there is $x\in X$ with $\|x\|=1$ such that $\|L(x)-\lambda x\|<\delta$.
Since the sequence
$(\lambda^nx)_{n\in\mathbb{Z}}$ satisfies
$\|L(\lambda^nx)-\lambda^{n+1}x\|=\|\lambda^n(L(x)-\lambda x)\|=\|L(x)-\lambda x\|<\delta$ for all $n\in\mathbb{Z}$, there is $y\in X$ such that
$\|L^n(y)-\lambda^nx\|\leq \epsilon$ for all $n\in\mathbb{Z}$.
Then,
\begin{equation}
\label{traidor}
\frac{1}2\leq \|L^n(y)\|\leq 1+\frac{1}2<2,\quad\quad\forall n\in\mathbb{Z}.
\end{equation}
This implies at once that $L$ cannot be expansive.
Now suppose that $L$ is asymptotically expansive and let $e$ be the asymptotic expansivity constant.
It follows that
$\|L^n(\frac{2}3ey)\|=\frac{2}3e\|L^n(y)\|\leq \frac{2}3e\frac{3}2=e$ for all $n\geq0$. Therefore,
$L^n(\frac{2}3ey)\to0$ as $n\to\infty$ hence $L^n(y)\to0$ too contradicting (\ref{traidor}).
This completes the proof.
\qed
\end{proof}

The next key lemma below is due to Bernardes and Messaoudi \cite{bm1}.

\begin{lem}
\label{trump}
If $L\in GL(X)$ is uniformly expansive, then there are $\beta>0$ and $c>0$ such that
$$
E^{cs}=\{x\in X:\|L^n(x)\|\leq c\beta^n\|x\|,\quad\forall n\geq0\}.
$$
\end{lem}

\begin{proof}
Obviously the set in the right of the above equality is contained the left-hand one. Hence we need to prove the converse contention. We proceed as follows.

Let $m\in \mathbb{N}$ be given by the definition of uniformly expansivity for $L$.
It follows that $\{\|x\|=1\}=A\cup B$ where
$A=\{x\in X:\|x\|=1, \|L^m(x)\|\geq2\}$ and $B=\{x\in X:\|x\|=1, \|L^{-m}(x)\|\geq2\}$.

If $x\in A$, then $\frac{L^m(x)}{\|L^m(x)\|}\in A$.
Otherwise, $\frac{L^m(x)}{\|L^m(x)\|}\in B$ so
$\frac{\|x\|}{\|L^m(x)\|}\geq2$ thus $1=\|x\|\geq 2\|L^m(x)\|\geq4$ an absurd.
Likewise, if $x\in B$, then $\frac{L^{-m}(x)}{\|L^{-m}(x)\|}\in B$.
Define
$\beta=(\frac{1}2)^m$ and $c=\max\{\frac{L^j}{\beta^{j}}:0\leq j\leq m\}$.

Now take $x\in E^{cs}$, we can assume $x\neq0$, and $n\in \mathbb{N}$.
Define two sequences
$y_k,z_k$ by
$y_1=z_1=\frac{L^n(x)}{\|L^n(x)\|}$, 
$$
y_k=\frac{L^m(y_{k-1})}{\|L^m(y_{k-1})\|}
\quad\mbox{ and }\quad z_k=\frac{L^{-m}(z_{k-1})}{\|L^{-m}(z_{k-1})\|},
\quad\forall k\geq2.
$$
It follows that
\begin{equation}
\label{parler1}
y_k=\frac{L^{(k-1)m+n}(x)}{\|L^m(y_{k-1})\|\cdots \|L^m(y_2)\|\|L^n(x)\|}
\end{equation}
and
\begin{equation}
\label{parler2}
z_k=\frac{L^{-(k-1)m+n}(x)}{\|L^{-m}(z_{k-1})\|\cdots \|L^{-m}(z_2)\|\|L^n(x)\|}.
\end{equation}
If $y_1\in A$, then $y_k\in A$ for all $k\geq0$ so 
$$
2\leq \|L^m(y_k)\|=\frac{\|L^{km+n}(x)\|}{\|L^m(y_{k-1})\|\cdots \|L^m(y_2)\|\|L^n(x)\|}
$$
thus
$$
\|L^{km+n}(x)\|\geq 2\|L^m(y_{k-1})\|\cdots \|L^m(y_2)\|\|L^n(x)\|\geq 2^{k-1}\|L^n(x)\|
$$
hence $\|L^{km+n}(x)\|\to\infty$ as $k\to\infty$ contradicting $x\in E^{cs}$.

Then, $y_1\in B$ so $z_1=y_1\in B$ thus $z_k\in B$ for every $k\geq0$.
From this and (\ref{parler2}) we obtain
$$
2\leq \|L^{-m}(z_k)\|=\frac{\|L^{-km+n}(x)\|}{\|L^{-m}(z_{k-1})\|\cdots \|L^{-m}(z_2)\|\|L^n(x)\|}
$$
thus
$$
\|L^{-km+n}(x)\|\geq 2\|L^{-m}(z_{k-1})\|\cdots \|L^{-m}(z_2)\|\|L^n(x)\|\geq 2^{k-1}\|L^n(x)\|
$$
for all $k,n\geq0$.
Replacing $n$ by $km$ we get
$
\|x\|\geq 2^{k-1}\|L^{km}(x)\|
$
i.e.
$$
\|L^{km}(x)\|\leq \beta^{k-1}\|x\|, \quad\quad\forall k\geq0, x\in E^{cs}.
$$
From this we get $\|L^n(x)\|\leq c\beta^n\|x\|$ for all $x\in E^{sc}$ proving the result.
\qed
\end{proof}

\begin{proof}[of Theorem \ref{thA-1}]
That Item (3) implies Item (1) is already known \cite{eh,o}.
Conversely, suppose that $L$ is expansive and has the shadowing property.
Then, $L$ is uniformly expansive by Lemma \ref{l2-1} and so Lemma \ref{trump} provides $c>0$ and $0< \beta<1$ such that
$$
E^{cs}=\{x\in X:\|L^m(x)\|\leq c\beta^n\|x\|, \forall n\geq0\}.
$$
This implies both $E^{cs}$ is closed and $r(L|_{E^{cs}})<1$ by Gelfand's spectral radius formula.
Applying the same argument to $L^{-1}$ we have both $E^{cu}$ is closed and $r(L^{-1}|_{E^{cu}})<1$.
Since $L$ has the shadowing lemma, $X=E^{cs}+E^{cu}$ by lemma \ref{l1-1} and, since $L$ is expansive, we also have
$E^{cs}\cap E^{cu}=\{0\}$. All together imply that $L$ is hyperbolic  proving Item (3).
We conclude that items (1) and (3) are equivalent.

Now suppose that Item (2) holds i.e. $L$ is asymptotically expansive and has the shadowing property.
By Lemma \ref{l2-1} we have that $L$ is uniformly expansive (hence expansive) and so Item (1) holds.
Finally, suppose that Item (3) holds (in particular, $L$ is uniformly expansive). If $x\in X$ and $\|L^n(x)\|\leq1$ for $n\geq0$, then $x\in E^{cs}$.
Since $L$ is uniformly expansive, we can apply Lemma \ref{trump} to get $c>0$ and $0<\beta<1$ such that
$\|L^n(x)\|\leq c\beta^n\|x\|$ for $n\geq0$ whence $L^n(x)\to0$ as $n\to\infty$. Therefore, $L$ is asymptotically expansive proving Item (2). This completes the proof.
\qed
\end{proof}

\chapter{Generalized hyperbolicity}

Before introducing the definition of generalized hyperbolic operator we will present some
motivating results. The first one is as follows.

\begin{prop}
\label{motiv1}
Let $X$ be a Banach space and $L\in GL(X)$.
If $X=A\oplus B$ is a splitting (direct sum) by closed subspaces $A$ and $B$
such that $L(A)\subset A$ and $L(B)\subset B$, then $L(A)=A$ and $L(B)=B$.
\end{prop}

\begin{proof}
First we show $B\subset L(B)$.
Suppose not i.e. there is $x\in B\setminus L(B)$.
Let $P_A:X\to A$ and $P_B:X\to B$ be the associated projections.
Then, $P_AL^{-1}(x)\neq0$ and also
$L^{-1}(x)=P_AL^{-1}(x)+P_BL^{-1}(x)$.
It follows that
$LP_AL^{-1}(x)=x-LP_BL^{-1}(x)\in B$ hence
$LP_AL^{-1}(x)\in B\cap A=\{0\}$ therefore
$P_AL^{-1}(x)=0$ which is absurd.
Then, $B\subset L(B)$ and so $L(B)=B$.
Reversing the roles of $A$ and $B$ we get $L(A)=A$ completing the proof.
\qed
\end{proof}

The next is a useful concept from \cite{mn} based on \cite{bcdmp}.

\begin{deff}
A number $K\in ]0,\infty[$ is called
{\em shadowableness constant} of $L\in GL(X)$ if for every bounded sequence
$(z_n)_{n\in\mathbb{Z}}$ of $X$ there is a sequence $(y_n)_{n\in\mathbb{Z}}$ such that
$$
\sup_{n\in\mathbb{Z}}\|y_n\|\leq K\sup_{n\in\mathbb{Z}}\|z_n\|\quad\mbox{ and }\quad y_{n+1}=Ly_n+z_n,\quad\forall n\in\mathbb{Z}.
$$
\end{deff}
We define
$$
Shad(L)=\inf\{ K>0:K\mbox{ is a shadowableness constant of }L\}
$$
(with the convention that $\inf\emptyset=\infty$).
The lemma below justifies the term "shadowableness constant" in this definition.

\begin{lem}
\label{l1-2}
$L\in GL(X)$ has the shadowing property $\iff$ $Shad(L)<\infty$.
\end{lem}

\begin{proof}
First assume that $L$ has the shadowing property and take $\delta>0$ from this property for $\epsilon=1$. We shall prove that $K=\frac{1}\delta$ is a shadowableness constant.
Let $(z_n)_{n\in\mathbb{Z}}$ be a bounded sequence and define $M=\sup_{n\in\mathbb{Z}}\|z_n\|$. We can assume that $M\neq0$. Take a sequence
$(x_n)_{n\in\mathbb{Z}}$ such that
$x_{n+1}=Lx_n+\frac{\delta}Mz_n$ for $n\in\mathbb{Z}$
(this sequence is completely determined by $x_0$).
Clearly $\|Lx_n-x_{n+1}\|\leq\delta$ for $n\in \mathbb{Z}$ hence there is $x\in X$ such that
$\|L^n(x)-x_n\|\leq 1$ for all $n\in\mathbb{Z}$.
Define $y_n=\frac{M}\delta(x_n-L^nx)$ for $n\in \mathbb{Z}$.
Then,
$\|y_n\|\leq \frac{M}\delta=K\sup_{n\in\mathbb{Z}}\|z_n\|$ for all $n\in\mathbb{Z}$ hence
$$
\sup_{n\in\mathbb{Z}}\|y_n\|\leq K\sup_{n\in\mathbb{Z}}\|Z_n\|.
$$
Also
$$
y_{n+1}=\frac{M}\delta(x_{n+1}-L^{n+1}x)=\frac{M}\delta(Lx_n+\frac{\delta}Mz_n-L^{n+1}x)=Ly_n+z_n,
$$
so $y_{n+1}=L(y_n)+z_n$ for all $n\in\mathbb{Z}$. Therefore, $K$ is a shadowableness constant.

Conversely,
suppose that there is a shadowableness constant $K$. Fix the given $\epsilon>0$ we set
$\delta=\frac{\epsilon}K$. Let $(x_n)_{n\in\mathbb{Z}}$ such that
$\|Lx_n-x_{n+1}\|\leq \delta$ for $n\in\mathbb{Z}$.
Take $z_n=Lx_n-x_{n+1}$ so $(z_n)_{n\in\mathbb{Z}}$ satisfies
$\sup_{n\in \mathbb{Z}}\|z_n\|\leq\delta$ and so it is bounded.
Then, there is $(y_n)_{n\in\mathbb{Z}}$ such that
$\sup_{n\in\mathbb{Z}}\|y_n\|\leq K\delta$ and $y_{n+1}=Ly_n+z_n$.
Set $x=y_0-x_0$.
Since
$y_{n+1}=Ly_n+Lx_n-x_{n+1}$, $y_{n+1}-x_{n+1}=L(y_n+x_n)=\cdots =L^{n+1}(x)$ for all $n\in \mathbb{Z}$ hence
$$
L^nx=y_n+x_n, \quad\quad\forall n\in\mathbb{Z}.
$$
Then,
$$
\|L^nx-x_n\|=\|y_n\|\leq \sup_{n\in\mathbb{Z}}\|y_n\|\leq K\delta=\epsilon,\quad\quad\forall n\in\mathbb{Z}
$$
proving the shadowing property for $L$. This ends the proof.
\qed
\end{proof}

Now we establish some properties of $Shad(L)$.
If $X_1,X_2$ are (complex) vector spaces the product $X_1\times X_2$ is a vector space
with the standard operations
$(x_1,x_2)+(x_1',x_2')=(x_1+x_1',x_2+x_2')$ and $\lambda(x_1,x_2)=(\lambda x_1,\lambda x_2)$
for $(x_1,x_2),(x_1',x_2')\in X_1\times X_2$ and $\lambda\in\mathbb{C}$.
If additionally $X_1$ and $X_2$ are Banach spaces, then so does
$X_1\times X_2$ if endowed with the norm $\|(x_1,x_2)\|=\max\{\|x_1\|,\|x_2\|\}$.

\begin{lem}
\label{arab}
$\,$
\begin{enumerate}
\item If $X$ and $Y$ are Banach spaces, $L\in GL(X)$ and $H:X\to Y$ is a linear homeomorphism, then
$$
(\|H\|\|H^{-1}\|)^{-1} Shad(L)\leq Shad(H^{-1}\circ L\circ H)\leq\|H\|\|H^{-1}\| Shad(L).
$$
\item
If $X_1,X_2$ are Banach spaces and $L_i\in GL(X_i)$ for $i=1,2$,
then
$$
Shad(L_1\times L_2)=\max\{Shad(L_1),Shad(L_2)\}.
$$
\item
If $L\in GL(X)$, then
$$
\|L^{-1}\|^{-1}Shad(L^{-1})\leq Shad(L)\leq \|L\|Shad(L^{-1}).
$$
\end{enumerate}
\end{lem}

\begin{proof}
To prove Item (1) we first prove
\begin{equation}
\label{nepo}
Shad(H^{-1}\circ L\circ H)\leq \|H\|\|H^{-1}\| Shad(L).
\end{equation}
If $Shad(L)=\infty$ nothing to prove so we can assume $Shad(L)<\infty$.
Let $\epsilon>0$ then there is a shadowableness constant $K< Shad(L)+\epsilon$ for $L$.
Take a bounded sequence $(x_n)_{n\in\mathbb{Z}}$ in $X$.
Then, $(H(x_n))_{n\in\mathbb{Z}}$ is bounded in $Y$ so there is
a sequence $(y_n)_{n\in\mathbb{Z}}$ in $Y$ such that
$$
\sup_{n\in\mathbb{Z}}\|y_n\|\leq K\sup_{n\in\mathbb{Z}}\|H(x_n)\|
\quad\mbox{ and }\quad y_{n+1}=L(y_n)+H(x_n),\quad\quad\forall n\in\mathbb{Z}.
$$
Then, the sequence $(H^{-1}(y_n))_{n\in\mathbb{Z}}$ satisfies
$$
\sup_{n\in\mathbb{Z}}\|H^{-1}(y_n)\|\leq \|H^{-1}\| K\sup_{n\in\mathbb{Z}}\|H(x_n)\|\leq
\|H\| \|H^{-1}\| K\sup_{n\in\mathbb{Z}}\|x_n\|
$$
and
$$
H^{-1}(y_{n+1})=H^{-1}(L(y_n)+H(x_n))=(H^{-1}\circ L\circ H)(H^{-1}(y_n))+x_n,\quad\quad\forall n\in\mathbb{Z}
$$
proving that $\|H\|\|H^{-1}\| K$ is a shadowableness constant of $H^{-1}\circ L\circ H$.
It follows that
$$
Shad(H^{-1}\circ L\circ H)\leq \|H\|\|H^{-1}\|K\leq \|H\|\|H^{-1}\|(Shad(L)+\epsilon).
$$
Letting $\epsilon\to0$ we get \eqref{nepo}.
From this we obtain
$$
Shad(L)=Shad(H\circ (H^{-1}\circ L\circ H)\circ H^{-1})\leq \|H\|\|H^{-1}\| Shad(H^{-1}\circ L\circ H)
$$
proving Item (1).

To prove Item (2) we first show
\begin{equation}
\label{f4}
\max\{Shad(L_1),Shad(L_2)\}\leq Shad(L_1\times L_2).
\end{equation}
Again if $Shad(L_1\times L_2)=\infty$ nothing to prove.
Otherwise,
take $\epsilon>0$ and a shadowableness constant $K\leq Shad(L_1\times L_2)+\epsilon$.
Let $(z^1_n)_{n\in\mathbb{Z}}$ be a bounded sequence in $X^1$.
Then, $((z^1_n,0))_{n\in\mathbb{Z}}$ is bounded in $X^1\times X^2$ so there is $((x^1_n,x^2_n))_{n\in\mathbb{Z}}$ such that
$$
\sup_{n\in\mathbb{Z}}\|(x^1_n,x^2_n)\|\leq K\sup_{n\in\mathbb{Z}}\|(z^1_n,0)\|
\mbox{ and }
(x^1_{n+1},x^2_{n+1})=(L_1\times L_2)(x^1_n,x^2_n)+(z^1_n,0),
$$
for all $n\in\mathbb{Z}.$
It follows that
$$
\sup_{n\in\mathbb{Z}}\|x^1_n\|\leq \sup_{n\in\mathbb{Z}}\|(x^1_n,x^2_n)\|\leq K\sup_{n\in\mathbb{Z}}\|(z^1_n,0)\|=K\sup_{n\in\mathbb{Z}}\|z^1_n\|
$$
and $x^1_{n+1}=L_1(x_n)+z^1_n$ for all $n\in\mathbb{Z}$ proving
that $K$ is a shadowableness constant of $L_1$.
Then,
$Shad(L_1)\leq K<Shad(L_1\times L_2)+\epsilon$.
Letting $\epsilon\to0$ we get
$Shad(L_1)\leq Shad(L_1\times L_2)$.
Likewise, $Shad(L_2)\leq Shad(L_1\times L_2)$ proving \eqref{f4}.
Finally, we prove
\begin{equation}
\label{f5}
Shad(L_1\times L_2)\leq \max\{Shad(L_1),Shad(L_2)\}.
\end{equation}
If $Shad(L_i)=\infty$ for some $i=1,2$ nothing to prove.
Then, we can assume $Shad(L_i)<\infty$ for all $i=1,2$.
Fix $\epsilon>0$ and take
shadowableness constants $K_i<Shad(L_i)+\epsilon$ of $L_i$ for $i=1,2$.
Let $((z^1_n,z^2_n))_{n\in\mathbb{Z}}$ be bounded in $X_1\times X_2$.
Then, $(z^i_n)_{n\in\mathbb{Z}}$ is bounded in $X_i$ for $i=1,2$ so there is
$(y^i_n)_{n\in\mathbb{Z}}$ in $X_i$ such that
$$
\sup_{n\in\mathbb{Z}}\|y^i_n\|\leq K_i\sup_{n\in\mathbb{Z}}\|z^i_n\|
\quad\mbox{ and }\quad
x^i_{n+1}=L_i(x_n)+z^i_n,\quad\forall n\in\mathbb{Z}, i=1,2.
$$
It follows that
\begin{eqnarray*}
\sup_{n\in\mathbb{Z}}\|(y^1_n,y^2_n)\| & = & \sup_{n\in\mathbb{Z}}\max\{\|y^1_n\|,\|y^2_n\|\} \\
& \leq & \max\{K_1,K_2\}\sup_{n\in\mathbb{Z}}\max\{\|z^1_n\|,\|z^2_n\|\} \\
& = & \max\{K_1,K_2\}\sup_{n\in\mathbb{Z}}\|(z^1_n,z^2_n)\|
\end{eqnarray*}
and
$$
(y^1_{n+1},y^2_{n+1})=(L_1(y^1_n)+z^1_n,L_2(y^2_n)+z^2_n)=(L_1\times L_2)(y^1_n,y^2_n)+(z^1_n,z^2_n),
$$
for every $n\in\mathbb{Z}$.
This proves that $\max\{K_1,K_2\}$ is a shadowableness constant of $L_1\times L_2$
so
$$
Shad(L_1\times L_2)\leq \max\{K_1,K_2\}\leq \max\{Shad(L_1)+\epsilon,Shad(L_2)+\epsilon\}.
$$
Letting $\epsilon\to0$ we get \eqref{f5}. Then, Item (2) holds.

To prove Item (3) let $\hat{K}$ be a shadowableness constant of $L$.
We shall prove that $K=\|L\|\hat{K}$ is a shadowableness constant of $L^{-1}$.
If $(z_n)_{n\in\mathbb{Z}}$ is bounded, so does
$(-L(z_{-n-1}))_{n\in\mathbb{Z}}$ then there is a sequence
$(\hat{y}_n)_{n\in\mathbb{Z}}$ such that
$$
\sup_{n\in\mathbb{Z}}\|\hat{y}_n\|\leq \hat{K}\sup_{n\in\mathbb{Z}}\|L(z_{-n-1})\|
\quad\mbox{ and }\quad
\hat{y}_{n+1}=L(\hat{y}_n)-L(z_{-n-1}),\quad\forall n\in\mathbb{Z}.
$$
Defining
$y_n=\hat{y}_{-n}$ for $n\in\mathbb{Z}$ we get a sequence $(y_n)_{n\in\mathbb{Z}}$ such that
$$
\sup_{n\in\mathbb{Z}}\|y_n\|\leq K\sup_{n\in\mathbb{Z}}\|z_n\|
$$
and
$$
y_{n+1}=\hat{y}_{-n-1}=L^{-1}(\hat{y}_{-n})+z_n=L^{-1}(y_n)+z_n,\quad\forall n\in\mathbb{Z}
$$
proving that $K$ is a shadowableness constant of $L^{-1}$
It follows that
$$
Shad(L^{-1})\leq \|L\|Shad(L).
$$
It follows that
$
Shad(L)=Shad((L^{-1})^{-1})\leq \|L^{-1}\| Shad(L^{-1})
$
so
$$
\|L^{-1}\|^{-1}Shad(L^{-1})\leq Shad(L^{-1}).
$$
This completes the proof.
\qed
\end{proof}

An application of the above lemma is the following result.

\begin{theo}
\label{agripino}
Let $X$ a Banach space and $L\in GL(X)$.
If there is a direct sum $X=X_1\oplus X_2$ by closed subspaces $X_1,X_2$ with
$L(X_i)=X_i$ for $i=1,2$, then $L$ has the shadowing property if and only if $L|_{X_1}$ and $L|_{X_2}$ do.
\end{theo}

\begin{proof}
Let $P_i:X\to X_i$ be the projections ($i=1,2$) associated to the splitting $X=X_1\oplus X_2$.
Define $H:X\to X_1\times X_2$ by $H(x)=(P_1(x),P_2(x))$.
Clearly $H$ is a linear map.
Since $X=X_1\oplus X_2$ we can see that $H$ is bijective with inverse
$H^{-1}(x_1,x_2)=x_1+x_2$ for $(x_1,x_2)\in X_1\times X_2$.
Then,
$$
\|H^{-1}\|\leq 2.
$$
Also, since
$$
\|H(x)\| = \max\{\|P_1(x)\|,\|P_2(x)\|\} \leq \max\{\|P_1\|,\|P_2\|\}\|x\|,
\quad\quad\forall x\in X,
$$
$H$ is a bounded operator with
$$
\|H\|\leq \max\{\|P_1\|,\|P_2\|\}.
$$
 Then, $H$ is a linear homeomorphism.
On the other hand,
\begin{eqnarray*}
H^{-1}\circ (L|_{X_1}\times L|_{X_2})\circ H(x) & = & H^{-1}((L|_{X_1}\times L|_{X_2})(P_1(x),P_2(x)))\\
& = & H^{-1}(L(P_1(x)),L(P_2(x)))\\
& = & L(P_1(x))+L(P_2(x))\\
& = & L(P_1(x)+P_2(x))\\
& = & L(x),\quad\quad\forall x\in X
\end{eqnarray*}
proving
$$
H^{-1}\circ (L|_{X_1}\times L_{X_2})\circ H=L.
$$
Then,
$$
\max\{Shad(L|_{X_1}), Shad(L|_{X_2})\}\leq 2\max\{\|P_1\|,\|P_2\|\}Shad(L)
$$
and
$$
Shad(L)\leq 2\max\{\|P_1\|,\|P_2\|\} \max\{Shad(L|_{X_1}), Shad(L|_{X_2})\}
$$
by Lemma \ref{arab}.
Therefore, $L$ has the shadowing property $\iff$ $Shad(L)<\infty$ (by Lemma \ref{l1-2})
$\iff$ $Shad(L|_{X_i})<\infty$ for $i=1,2$ $\iff$ $L|_{X_i}$ has the shadowing property for $i=1,2$.
This completes the proof.
\qed
\end{proof}

\begin{remark}
\label{the}
Since the identities $L(X_1)=X_1$ and $L(X_2)=X_2$ hold,
$(L(P_1(x)),L(P_2(x)))\in X_1\times X_2$
so $H^{-1}(L(P_1(x)),L(P_2(x)))= L(P_1(x))+L(P_2(x))$.
The question arise what if we change these identities by inclusions.
\end{remark}

In the following result we estimate the shadowableness constant.

\begin{theo}
\label{teo}
Let $X$ be a Banach space and $L\in GL(X)$.
If $X=S\oplus U$ is a splitting into closed subsets $S$, $U$ of $X$ such that
$L(S)\subset S$ and $L^{-1}(U)\subset U$, then
\begin{equation}
\label{kgb}
Shad(L)\leq \|P_S\|\displaystyle\sum_{k=0}^\infty\|L^k|_S\|+\|P_U\|\displaystyle\sum_{k=1}^\infty\|L^{-k}|_U\|,
\end{equation}
where $P_S:X\to S$ and $P_U:X\to U$ are the projections associated to the splitting $X=S\oplus U$.
If additionally $\displaystyle\sum_{k=0}^\infty L^k|_S$ and $\displaystyle\sum_{k=1}^\infty L^{-k}|_U$ are strongly convergent, then
\begin{equation}
\label{egito}
\max\left\{\left\|\displaystyle\sum_{k=0}^\infty L^k|_S\right\|,
\left\|\displaystyle\sum_{k=0}^\infty L^{-k}|_U\right\|\right\}\leq Shad(L).
\end{equation}
\end{theo}

\begin{proof}
To prove the first inequality we can assume
$$
A=\displaystyle\sum_{k=0}^\infty \|L^k|_A\|<\infty\quad\quad\mbox{ and }\quad\quad B=\displaystyle\sum_{k=1}^\infty\|L^{-k}|_B\|<\infty.
$$
We assert that
$K=\|P_S\|A+\|P_U\| B$
is a shadowableness constant of $L$.

Take a bounded sequence
$(z_n)_{n\in\mathbb{Z}}$ and define the sequence
$(y_n)_{n\in\mathbb{Z}}$ by 
$$
y_n=\displaystyle\sum_{k=0}^\infty L^kP_S(z_{n-k-1})-\displaystyle\sum_{k=1}^{\infty}L^{-k}P_U(z_{n+k-1}),\quad\quad\forall n\in\mathbb{Z}.
$$
We have
\begin{eqnarray*}
\|y_n\| & \leq & \displaystyle\sum_{k=0}^\infty \|L^k|_S\|\|P_S(z_{n-k-1})\|+\displaystyle\sum_{k=1}^\infty \|L^{-k}|_U\| \|P_U(z_{n+k-1})\| \\
& \leq & (A\|P_S\|+B\|P_U\|)\sup_{n\in\mathbb{Z}}\|z_n\| \\
& \leq & K\sup_{n\in\mathbb{Z}}\|z_n\|,\quad\quad\forall n\in\mathbb{Z}
\end{eqnarray*}
so
$$
\sup_{n\in\mathbb{Z}}\|y_n\|\leq K\sup_{n\in\mathbb{Z}}\|z_n\|.
$$
Also
\begin{eqnarray*}
y_{n+1}& = & \displaystyle\sum_{k=0}^\infty L^k P_S(z_{n-k})-\displaystyle\sum_{k=1}^\infty L^{-k}P_U(z_{n+k})\\
& = & \displaystyle\sum_{k=1}^\infty L^kP_S(z_{n-k})-\left(\displaystyle\sum_{k=1}^\infty L^{-k}P_U(z_{n+k})+P_U(z_n)\right)+z_n\\
& = &  \displaystyle\sum_{k=1}^\infty L^kP_S(z_{n-k})-\displaystyle\sum_{k=0}^\infty L^{-k}P_U(z_{n+k-1})+z_n\\
& = & L\left( \displaystyle\sum_{k=1}^\infty L^{k-1}P_S(z_{n-k})-\displaystyle\sum_{k=0}^\infty L^{-(k+1)}P_U(z_{n+k-1}) \right) + z_n \\
& = & L\left( \displaystyle\sum_{k=0}^\infty L^{k}P_S(z_{n-k-1})-\displaystyle\sum_{k=1}^\infty L^{-k}P_U(z_{n+k-1}) \right) + z_n \\
& = & Ly_n+z_n,\quad\quad\forall n\in\mathbb{Z}
\end{eqnarray*}
so $K$ is a shadowableness constant.
This proves \eqref{kgb}.

Now we assume that $\displaystyle\sum_{k=0}^\infty L^k|_S$ and $\displaystyle\sum_{k=1}^\infty L^{-k}|_U$ are strongly convergent. We can assume that $Shad(L)<\infty$.
Fix $\epsilon>0$, a shadowableness constant $K< Shad(L)+\epsilon$ and
$x\in S$. The constant sequence $(x)_{n\in\mathbb{Z}}$ is clearly bounded so
there is a sequence $(y_n)_{\in\mathbb{Z}}$ in $X$  with $y_0\in S$ such that
\begin{equation}
\label{higgs}
\sup_{n\in\mathbb{Z}}\|y_n\|\leq K\|x\|
\quad\quad\mbox{ and }\quad\quad y_{n+1}=L(y_n)+x,\quad\quad\forall n\in \mathbb{Z}.
\end{equation}
It follows that
$$
y_n=L^n(y_0)+\displaystyle\sum_{k=0}^{n-1} L^k(x), \quad\quad\forall n\in\mathbb{N}.
$$
But since $\displaystyle\sum_{k=0}^\infty L^k|_S$ is strongly
convergent, $L^n|S\to0$ strongly as $n\to\infty$ so $L^n(y_0)\to0$ thus
$$
\lim_{n\to\infty}y_n=\displaystyle\sum_{k=0}^\infty L^k(x).
$$
Then, the inequality in \eqref{higgs} implies
$$
\left\|
\displaystyle\sum_{k=0}^\infty L^k(x)
\right\|
=\|\lim_{n\to\infty}y_n\|=\lim_{n\to\infty}\|y_n\|
\leq \sup_{n\in\mathbb{Z}}\|y_n\|
\leq K\|x\|.
$$
Since $x\in S$ is arbitrary and $K< Shad(L)+\epsilon$,
$$
\left\|
\displaystyle\sum_{k=0}^\infty L^k|_S\right\|
\leq Shad(L)+\epsilon.
$$
Letting $\epsilon\to0$ above we get
$$
\left\|
\displaystyle\sum_{k=0}^\infty L^k|_S\right\|
\leq Shad(L).
$$
Likewise
$$
\left\|
\displaystyle\sum_{k=0}^\infty L^{-k}|_U\right\|
\leq Shad(L)
$$
so \eqref{egito} holds.
\qed
\end{proof}


Combining theorems \ref{agripino} and \ref{teo} we get the following corollary.

\begin{cor}
\label{zoom}
If $L\in GL(X)$ is {\em hyperbolic} i.e. there is a splitting $X=S\oplus U$ formed by closed subspaces $S$ and $U$ such that
$L(S)=S$, $L(U)=U$, $r(L|_S)<1$ and $r(L^{-1}|_U)<1$,
then $L$ has the shadowing property.
\end{cor}

Taking Remark \ref{the} into account one is tempted to
change the identities $L(S)=S$ and $L(U)=U$ in the above corollary by inclusions.
In light of Proposition \ref{motiv1} we have to consider the case
$L(S)\subset S$ and $U\subset L(U)$ only.
By doing so we get the following concept from \cite{cgp}.

\begin{deff}
We say that $L\in GL(X)$ is {\em generalized hyperbolic} if there is a splitting $X=S\oplus U$ into closed subspaces $S$ and $U$ such that $L(S)\subset S$, $L^{-1}(U)\subset U$,
$r(L|_S)<1$ and $r(L^{-1}|U)<1$.
\end{deff}

Every hyperbolic linear homeomorphism is therefore generalized hyperbolic.
We will postpone the example of a linear homeomorphism which is generalized hyperbolic but not hyperbolic till Example \ref{lockdown}.
Meanwhile we will obtain some similarities between these kind of operators.
Indeed, we shall prove that Corollary \ref{zoom} is valid for generalized hyperbolic operators.
However, the proof cannot use Theorem \ref{agripino}.

\begin{theo}
\label{thA-2}
If $L\in GL(X)$ is generalized hyperbolic with associated splitting $X=S\oplus U$ and projections
$P_S:X\to S$ and $P_U:X\to U$, then
$$
Shad(L)\leq\|P_S\|\displaystyle\sum_{k=0}^\infty \|L^k|_S\|+\|P_U\|\displaystyle\sum_{k=1}^\infty\|L^{-k}|_U\|.
$$
In particular, $L$ has the shadowing property.
\end{theo}

\begin{proof}
As in the proof of Theorem \ref{teo} we have
$$
A=\displaystyle\sum_{k=0}^\infty\|L^k|_S\|<\infty
\quad\mbox{ and }\quad
B=\displaystyle\sum_{k=0}^\infty\|L^{-k}|_U\|<\infty.
$$
We shall prove that
$$
K=A\|P_S\|+B\|P_U\|
$$
is a shadowableness constant of $L$.

Take a bounded sequence
$(z_n)_{n\in\mathbb{Z}}$ and define the sequence
$(y_n)_{n\in\mathbb{Z}}$ by 
$$
y_n=\displaystyle\sum_{k=0}^\infty L^kP_S(z_{n-k-1})-\displaystyle\sum_{k=1}^{\infty}L^{-k}P_U(z_{n+k-1}),\quad\quad\forall n\in\mathbb{Z}.
$$
We have
\begin{eqnarray*}
\|y_n\| & \leq & \displaystyle\sum_{k=0}^\infty \|L^k|_S\|\|P_S(z_{n-k-1})\|+\displaystyle\sum_{k=1}^\infty \|L^{-k}|_U\| \|P_U(z_{n+k-1})\| \\
& \leq & (A\|P_S\|+B\|P_U\|)\sup_{n\in\mathbb{Z}}\|z_n\| \\
& \leq & K\sup_{n\in\mathbb{Z}}\|z_n\|,\quad\quad\forall n\in\mathbb{Z}
\end{eqnarray*}
so
$$
\sup_{n\in\mathbb{Z}}\|y_n\|\leq K\sup_{n\in\mathbb{Z}}\|z_n\|.
$$
Also
\begin{eqnarray*}
y_{n+1}& = & \displaystyle\sum_{k=0}^\infty L^k P_S(z_{n-k})-\displaystyle\sum_{k=1}^\infty L^{-k}P_U(z_{n+k})\\
& = & \displaystyle\sum_{k=1}^\infty L^kP_S(z_{n-k})-\left(\displaystyle\sum_{k=1}^\infty L^{-k}P_U(z_{n+k})+P_U(z_n)\right)+z_n\\
& = &  \displaystyle\sum_{k=1}^\infty L^kP_S(z_{n-k})-\displaystyle\sum_{k=0}^\infty L^{-k}P_U(z_{n+k-1})+z_n\\
& = & L\left( \displaystyle\sum_{k=1}^\infty L^{k-1}P_S(z_{n-k})-\displaystyle\sum_{k=0}^\infty L^{-(k+1)}P_U(z_{n+k-1}) \right) + z_n \\
& = & L\left( \displaystyle\sum_{k=0}^\infty L^{k}P_S(z_{n-k-1})-\displaystyle\sum_{k=1}^\infty L^{-k}P_U(z_{n+k-1}) \right) + z_n \\
& = & Ly_n+z_n,\quad\quad\forall n\in\mathbb{Z}
\end{eqnarray*}
so $K$ is a shadowableness constant.
Therefore, $L$ has the shadowing property by Lemma \ref{l1-2} and the proof follows.
\qed
\end{proof}

The converse of Theorem \ref{thA-2} was questioned by D'Aniello et al \cite{ddm}. More precisely,
they proposed the following question.

\begin{ques}
\label{q1}
Is every linear homeomorphism with the shadowing property of a Banach space generalized hyperbolic?
\end{ques}

About this question we can make the following remark.

\begin{rem}
It follows from the proof od Theorem \ref{thA-2} that every $L\in GL(X)$
exhibiting a splitting $X=S\oplus U$ by closed subspaces $S$ and $U$ such that $L(S)\subset S$, $L^{-1}(U)\subset U$;
$$
\displaystyle\sum_{k=0}^\infty\|L|_S\|<\infty\quad\mbox{ and }\quad \displaystyle\sum_{k=0}^\infty\|L^{-1}|_U\|<\infty
$$
has the shadowing property. For the moment such homeomorphisms will be referred to as {\em weakly hyperbolic}.
It is therefore tempting to obtain negative answer for the question as soon as we find a weakly hyperbolic linear homeomorphism which is not generalized hyperbolic.
Nevertheless, both concepts of hyperbolicity for linear homeomorphisms are indeed equivalent.
\end{rem}

Let us present a partial positive answer for the question.
We follow part of the recent work \cite{lm}.
To motivate
we first give a consequence of Theorem \ref{thA-2}.
Recall that $L\in GL(X)$ is {\em equicontinuous} if for every $\epsilon>0$ there is $\delta>0$ such that
if $x\in X$ and $\|x\|\leq\delta$, then $\|L^n(x)\|\leq\epsilon$ for all $n\in\mathbb{Z}$.
The {\em central direction} of $L$ is defined by $E^c=E^{cs}\cap E^{cu}$ namely
$$
E^c=\{x\in X:\sup_{n\in\mathbb{Z}}\|L^nx\|<\infty\}.
$$
This is clearly a linear subspace and $L(E^c)=E^c$. Moreover,
$L$ is expansive if and only if $E^c=\{0\}$ In particular, if $L$ is expansive, then $E^c$ is a closed subspace.
The converse is true for generalized hyperbolic operators namely the following result holds.

\begin{theo}
\label{thB-2}
If $L\in GL(X)$ is generalized hyperbolic and $E^c$ is closed, then $E^c=\{0\}$.
In particular, $L$ is expansive (hence hyperbolic).
\end{theo}

\begin{proof}
First we note that
$$
\sup_{n\in\mathbb{Z}}\|L^n(z^c)\|<\infty,\quad\quad\forall z^c\in E^c.
$$
So, the Banach-Steinhouse theorem implies
$$
\sup_{n\in\mathbb{Z}}\|(L|_{E^c})^n\|<\infty.
$$
Then, it follows that $L|_{E_c}$ is equicontinuous.
Now, let $P_S:X\to S$ and $P_U:X\to U$ be the projections
associated to the generalized hyperbolic splitting $X=S\oplus U$.
Given $z^c\in E^c$ we write
$z^c=P_Sz^c+P_Uz^c$.
Since
$\|L^{-n}(z^c-P_Sz^c)\|=\|L^{-n}P_Uz^c\|\to0$ as $n\to\infty$,
we have
$\sup_{n\geq 0}\|L^{-n}(z^c-P_Sz^c)\|<\infty$. Also
$$
\sup_{n\geq0}\|L^n(z^c-P_Sz^c)\|\leq\sup_{n\geq0}\|L^nz^c\|+\sup_{n\geq0}\|L^nP_Sz^c\|<\infty
$$
so
$\sup_{n\in\mathbb{Z}}\|L^n(z^c-P_Sz^c)\|<\infty$ proving
$$
P_Uz^c=z^c-P_Sz^c\in E^c.
$$
Likewise, $P_Sz^c\in E^c$ proving
$$
E^c=(E^c\cap S)\oplus (E^c\cap U).
$$
Clearly $E\cap S$ and $E\cap U$ are closed subspaces, $L(E^c\cap S)\subset E^c\cap S$, $L^{-1}(E^c\cap U)\subset E^c\cap U$, $r(L|_{E^c\cap S})<1$ and $r(L|_{E^c\cap U})<1$.
All together imply that $L|_{E^c}\in GL(E^c)$ is generalized hyperbolic. Then, it has the shadowing property by Theorem \ref{thA-2}.
This implies that $E^c=\{0\}$.
Therefore, $L$ is expansive (see the remark before Lemma \ref{l1-1}) and  the proof follows.
\qed
\end{proof}

In light of Question \ref{q1}, it is natural to think that Theorem \ref{thB-2} holds for linear homeomorphisms with the shadowing property too.
Indeed, we will prove that this is the case.

\begin{theo}
\label{thC-2}
If $L\in GL(X)$ has the shadowing property and $E^c$ is closed, then $E^c=\{0\}$.
In particular, $L$ is expansive (hence hyperbolic).
\end{theo}

At first glance we can try to prove this result by repeating the proof of Theorem \ref{thB-2} with the
generalized hyperbolic hypothesis replaced by shadowing.
Indeed, everything go well till the assertion that $L|_{E^c}$ has the shadowing property.
When $L$ is generalized hyperbolic this follows because the generalized hyperbolicity of $L$ implies that of $L|_{E^c}$ (and then the shadowing of $L|_{E^c}$ follows from Theorem \ref{thA-2}). In the shadowing case we cannot assert that the shadowing of $L$ implies that of $L|_{E^c}$. Instead, we will prove that the shadowing property of $L$ will induce a related property on $L|_{E^c}$. This is the idea behind the next definition.

\begin{deff}
A homeomorphism of a metric space $g:Y\to Y$ has the {\em bounded shadowing property} if for every $\epsilon>0$ there is $\delta>0$ such that for every {\em bounded} sequence $(x_n)_{n\in\mathbb{Z}}$ with
$d(g(x_n),x_{n+1})\leq\delta$ for $n\in\mathbb{Z}$ there is $x\in Y$ such that
$d(g^n(x),x_n)\leq\epsilon$ for $n\in\mathbb{Z}$.
\end{deff}

In other words $g$ has the bounded shadowing property if only the {\em bounded} pseudo orbits can be shadowed.
Clearly both the shadowing and the bounded shadowing properties are equivalent on bounded metric spaces and further every homeomorphism with the shadowing property has the bounded shadowing property too.
We don't know however if these two properties are equivalent on {\em every} metric space.

The main motivation for this definition comes from the following elementary lemma.
It asserts that the shadowing property of a linear homeomorphisms induces the bounded shadowing property on its central direction. More precisely, we have the following result.

\begin{lem}
\label{l2-C}
If $X$ is a Banach space and $L\in GL(X)$ has the shadowing property, then $L|_{E^c}$ has the bounded shadowing property.
\end{lem}

\begin{proof}
Let $\epsilon>0$ be given and $\delta$ be from the shadowing property of $L$.
Let $(x^c_n)_{n\in\mathbb{Z}}$ be a bounded sequence in $E^c$ such that
$\|L(x^c_n)-x^c_{n+1}\|\leq \delta$ for every $n\in\mathbb{Z}$.
Then, the shadowing property provides $y\in X$ such that
$$
\|L^n(y)-x^c_n\|\leq \epsilon,\quad\quad\forall n\in\mathbb{Z}.
$$
It follows that
$\|L^n(y)\|\leq \epsilon+\|x^c_n\|$ and so
$$
\sup_{n\in\mathbb{Z}}\|L^n(y)\|\leq \epsilon+\sup_{n\in\mathbb{Z}}\|x^c_n\|<\infty.
$$
Then, $y\in E^c$ and so $L|_{E^c}$ has the bounded shadowing property.
\qed
\end{proof}

To apply this lemma to prove Theorem \ref{thC-2} we need the
following lemma about the invariance of the bounded shadowing property under topological conjugacy.

\begin{lem}
\label{l2-B}
Let $Y$, $Z$ be a Banach spaces and  $H:Z\to Y$ be a linear homeomorphism.
If $P\in GL(Z)$ has the bounded shadowing property, then so does $H\circ P\circ H^{-1}\in GL(Y)$.
\end{lem}

\begin{proof}
To prove the result we just need to assume that $H$ is a bi-Lipschitz homeomorphism
(in particular $H$ is uniformly continuous with uniformly continuous inverse $H^{-1}$).
Take $\epsilon>0$  and let $\epsilon'>0$ such that
$$
a,b\in Z\mbox{ and }\|a-b\|\leq \epsilon'\quad\quad\mbox{ implies }\quad\quad\|H(a)-H(b)\|\leq\epsilon.
$$
For this $\epsilon'$ we take $\delta'$ from the bounded shadowing property of $P$ and, for this $\delta'$, we take $\delta>0$ such that
$$
c,d\in Y\mbox{ and }\|c-d\|\leq \delta\quad\quad\mbox{ implies }\quad\quad\|H^{-1}c-H^{-1}d\|\leq \delta'.
$$
Now let $(y_n)_{n\in\mathbb{Z}}$ be a bounded sequence in $Y$
such that
$$
\|H\circ P\circ H^{-1}(y_n)-y_{n+1}\|\leq \delta,\quad\quad\forall n\in\mathbb{Z}.
$$
Then,
$\|P(H^{-1}(y_n))-H^{-1}(y_{n+1})\|\leq \delta'$ for $n\in\mathbb{Z}$. Since
$$
\|H^{-1}(y_n)\|\leq \|H^{-1}(y_n)-H^{-1}(0)\|
\leq Lip(H^{-1})\|y_n\|+\|H^{-1}(0)\|,\quad\quad\forall n\in\mathbb{Z}
$$
(where $Lip(H^{-1})$ is the Lipschitz constant of $H^{-1}$),
we also have
$$
\sup_{n\in\mathbb{Z}}\|H^{-1}(y_n)\|\leq Lip(H^{-1})\sup_{n\in\mathbb{Z}}\|y_n\|+\|H^{-1}(0)\|<\infty,
$$
hence $(H^{-1}(y_n))_{n\in\mathbb{Z}}$ is a bounded sequence of $Z$.
Then, the bounded shadowing property provides $z\in Z$ such that
$$
\|L^n(z)-H^{-1}(y_n)\|\leq \epsilon',\quad\quad\forall n\in\mathbb{Z}.
$$
It follows that
$\|H(L^n(z))-y_n\|\leq \epsilon$ and so $y=H(z)$ satisfies
$$
\|(H\circ L\circ H^{-1})(y)-y_n\|=\|H(L^n(z))-y_n\|\leq \epsilon,\quad\quad\forall n\in\mathbb{Z}
$$
proving the result.
\qed
\end{proof}

\begin{proof}[of Theorem \ref{thC-2}]
Let $X$ a Banach space and $L\in GL(X)$ with the shadowing property
such that $E^c$ is closed.
Then, $E^c$ is a closed subspace with $L(E^c)=E^c$.
In particular, $E^c$ is a Banach space with the induced norm and $L|_{E^c}\in GL(E^c)$.
Define the Banach space $Z=(E^c,\|\cdot\|)$ where $\|\cdot\|$ is the induced norm from $X$.
Define $P\in GL(Z)$ by $P=L|_{E^c}$. Since $L$ has the shadowing property, $P$ has the bounded shadowing property by Lemma \ref{l2-C}.

We now construct a second Banach space $Y$ as follows.
By the Banach-Steinhouse Theorem as in Theorem \ref{thB-2} we have
that
\begin{equation}
\label{battle}
\sup_{n\in\mathbb{Z}}\|(L|_{E^c})^n\|<\infty.
\end{equation}
Define the new norm $\|\cdot\|'$ in $E^c$ by
$$
\|z^c\|'=\sup_{n\in\mathbb{Z}}\|L^n(z^c)\|,\quad\quad\forall z^c\in E^c.
$$
Clearly $\|z^c\|\leq \|z^c\|'$ for $z^c\in E^c$ and also because of (\ref{battle}) there is $0<M<\infty$ such that
$\|z^c\|'\leq M\|z^c\|$ for all $z^c\in E^c$ namely
\begin{equation}
\label{japa}
\|z^c\|\leq \|z^c\|'\leq M \|z^c\|,\quad\quad\forall z^c\in E^c.
\end{equation}
Put $Y=(E^c,\|\cdot\|')$.  It follows from (\ref{japa}) that $Y$ is a Banach space.

Next we define $H:Z\to Y$ by as the identity of $E^c$.
It follows from (\ref{japa}) that $H$ is a linear homeomorphism.
Since $P$ has the bounded shadowing property, so does $H\circ P\circ H^{-1}$ by Lemma \ref{l2-B}.
Moreover,
$$
\|H\circ P\circ H^{-1}(z^c)\|'=\|L(z^c)\|'=\sup_{n\in\mathbb{Z}}\|L^n(L(z^c))\|=\sup_{n\in\mathbb{Z}}\|L^n(z^c)\|=\|z^c\|',
$$
for every $z^c\in Y=E^c$ so
$H\circ P\circ H^{-1}$ is a linear isometry of $Y$.
Therefore, $E^c=Y=\{0\}$ by Lemma \ref{l2-A} completing the proof.
\qed
\end{proof}

It was proved in [Axiom] that an equicontinuous linear homeomorphism of a Banach space has the shadowing property only if $X=\{0\}$. As a direct corollary of the above proof
we obtain the similar result for the bounded shadowing property. More precisely, we have the following result.

\begin{cor}
An equicontinuous linear homeomorphism of a Banach space $X$ has the bounded shadowing property if and only if $X=\{0\}$.
\end{cor}

Since every subspace of a finite dimensional Banach space is closed we also obtain the following well-known corollary from Theorem \ref{thC-2}.

\begin{cor}
Every linear homeomorphism with the shadowing property of a finite dimensional Banach space is 
expansive (hence hyperbolic).
\end{cor}

\chapter{Structural stability}

\noindent
Let $X$ be a Banach space.

\begin{deff}
We say that $L\in GL(X)$ is {\em strongly structural stable} if for every $\epsilon>0$ there is $\delta>0$ such that for every Lipschitz and bounded map $\beta:X\to X$ with
$\max\{\|\beta\|_\infty, Lip(\beta)\}<\delta$ there is a homeomorphism $H:X\to X$ such that
$$
\|H-I\|_\infty\leq \epsilon\quad\quad\mbox{ and }\quad\quad  H\circ L=(L+\beta)\circ H.
$$
\end{deff}

The main result of this chapter is the following one proved in \cite{bm2}.
\begin{theo}
\label{thA-3}
If $L\in GL(X)$ is generalized hyperbolic, then $L$ is strongly structural stable.
\end{theo}

\begin{proof}
Let $P_S:X\to S$ and $P_U:X\to U$ be the projections corresponding to the splitting $X=S\oplus U$.
Put
$d=\max\{\|P_S\|,\|P_U\|\}$.
Since $r(L|_S)<1$ and $r(L^{-1}|_U)<1$, the spectral radius theorem implies that there there are $c>0$ and $0<t<1$ such that
$$
\|L^k|_S\|\leq ct^k\quad\quad\mbox{ and }\quad\quad \|L^{-k}|_U\|\leq ct^k\quad\quad\forall k\geq0.
$$
Then, the series
\begin{equation}
\label{convergent}
A=\displaystyle\sum_{k=0}^\infty \|L^k|_S\|\quad\quad\mbox{ and }\quad\quad\displaystyle B=\sum_{k=0}^\infty\|L^{-1}|_U\|
\end{equation}
are convergent. Clearly $A+B>0$.

Given $\epsilon>0$ we choose $\delta=\frac{\epsilon}{d(A+B)}$.
Let $\beta:X\to X$ be bounded, Lipschitz with $\max\{\|\beta\|_\infty,Lip(\beta)\}<\delta$.
We can assume that $L+\beta:X\to X$ is a uniform homeomorphism.
Put $M=L+\beta$.

Define $C_b(X)$ as the set of bounded uniformly continuous maps $\phi:X\to X$.
We have that $C_b(X)$ is a Banach space if endowed with usual function operations and the supremum norm $\|\phi\|_\infty$.
To simplify the notation we will write $PQ$ instead of $P\circ Q$ when considering the composition of maps.
Given a uniform homeomorphism $R:X\to X$ and $\phi\in C_b(X)$ we define
\begin{equation}
\label{perrus}
\Phi(\phi)=\phi R-L\phi.
\end{equation}
Since
$$
\|\Phi(\phi)\|_\infty\leq \|\phi R\|_\infty+\|L\phi\|\leq (1+\|L\|)\|\phi\|_\infty,
$$
we have that $\Phi(\phi)$ is bounded for every $\phi\in C_b(X)$.
Also we can see that $\Phi(\phi)$ is uniformly continuous so $\Phi(\phi)\in C_b(X)$.
Hence we have a map $\Phi:C_b(X)\to C_b(X)$ which is clearly linear
and bounded with norm $\|\Phi\|_\infty\leq 1+\|L\|$.

Now given $\alpha\in C_b(X)$
we have
$$
\left\|
\displaystyle\sum_{k=0}^\infty L^kP_S\alpha R^{-k-1}
\right\|\leq \displaystyle\sum_{k=0}^\infty\|L^kP_S\alpha R^{-k-1}\|\leq
\left(\displaystyle\sum_{k=0}^\infty
\|(L|_S)^k\|\right)\|\alpha\|_\infty<\infty
$$
hence
the operator series below
$$
\displaystyle\sum_{k=0}^\infty L^kP_S\alpha R^{-k-1}
$$
is convergent.
Likewise the series
$$
\displaystyle\sum_{k=1}^\infty L^{-k}P_U\alpha R^{k-1}
$$
is convergent hence the map
$\Gamma(\alpha):X\to X$ defined by
$$
\Gamma(\alpha)=\displaystyle\sum_{k=0}^\infty L^kP_S\alpha R^{-k-1}-\displaystyle\sum_{k=1}^\infty L^{-k}P_U\alpha R^{k-1}
$$
is well defined.
Since
$$
\|\Gamma(\alpha)\|_\infty
\leq 
\left\|
\displaystyle\sum_{k=0}^\infty L^kP_S\alpha R^{-k-1}\right\|
+
\left\|
\displaystyle\sum_{k=1}^\infty L^{-k}P_U\alpha R^{k-1}
\right\|
\leq d(A+B) \|\alpha\|_\infty<\infty,
$$
$\Gamma(\alpha):X\to X$ is bounded.
We claim that $\Gamma(\alpha)$ is also uniformly continuous.
Without loss of generality we can assume $\|\alpha\|_\infty>0$.
Fix $\epsilon>0$ and $k_0\in\mathbb{N}$ such that
$$
\displaystyle\sum_{k=k_0}^\infty\|L^k|_{S}\|<\frac{\epsilon}{8d\|\alpha\|_\infty}\quad\mbox{ and }\quad
\displaystyle\sum_{k=k_0}^\infty\|L^{-k}|_U\|<\frac{\epsilon}{8d\|\alpha\|_\infty}.
$$
Since $\alpha$ is uniformly continuous, there is $\delta>0$ such that
$$
\|a-b\|<\delta\quad\quad\Longrightarrow\quad\quad \|\alpha(a)-\alpha(b)\|\leq \frac{\epsilon}{4d(A+B)}.
$$
Since $R$ is a uniform homeomorphism, there is $\delta'>0$ such that
$$
\|x-y\|<\delta'\quad\quad\Longrightarrow\quad\quad \|R^{k-1}(x)-R^{k-1}(y)\|<\delta,\quad\quad\forall -k_0\leq k\leq k_0.
$$
Then, if $x,y\in X$ and
$\|x-y\|<\delta$,
\begin{eqnarray*}
\|\Gamma(\alpha)x-\Gamma(\alpha)y\| & \leq & \left\|
\displaystyle\sum_{k=0}^\infty L^k P_S(\alpha R^{-k-1}(x)-\alpha R^{-k-1}(y))\right\| \\
& & +\left\|
\displaystyle\sum_{k=1}^\infty L^{-k} P_U(\alpha R^{k-1}(x)-\alpha R^{k-1}(y))\right\| \\
& \leq & d\displaystyle\sum_{k=0}^\infty\|L^k|_S\|\|\alpha R^{-k-1}(x)-\alpha R^{-k-1}(y)\| \\
& & + d\displaystyle\sum_{k=1}^\infty\|L^{-k}|_S\|\|\alpha R^{k-1}(x)-\alpha R^{k-1}(y)\| \\
& \leq & d\left(\displaystyle\sum_{k=k_0}^\infty\|L^k|_S\|\right)2\|\alpha\|_\infty + d\left(\displaystyle\sum_{k=k_0}^\infty\|L^{-k}|_U\|\right)2\|\alpha\|_\infty \\
& & + d\left( \displaystyle\sum_{k=0}^{k_0-1}\|L^k|_S\| \right) \frac{\epsilon}{4d(A+B)}\\
& & 
+d\left( \displaystyle\sum_{k=1}^{k_0-1}\|L^{-k}|_U\| \right) \frac{\epsilon}{4d(A+B)} \\
& \leq & \frac{\epsilon}4+\frac{\epsilon}4+\frac{\epsilon}4+\frac{\epsilon}4\\
& = & \epsilon
\end{eqnarray*}
proving that $\Gamma(\alpha)$ is uniformly continuous.
It follows that $\Gamma(\alpha)\in C_b(X)$ for $\alpha\in C_b(X)$ yielding a bounded linear map
$\Gamma:C_b(X)\to C_b(X)$ with norm
$$
\|\Gamma\|_\infty\leq d(A+B).
$$
Since
\begin{eqnarray*}
\Phi(\Gamma(\alpha)) & = & \Gamma(\alpha)R-L\Gamma(\alpha)\\
& = & \displaystyle\sum_{k=0}^\infty L^kP_S\alpha R^{-k}-\displaystyle\sum_{k=1}^\infty L^{-k}P_U\alpha R^k-\\
& & \left(\displaystyle\sum_{k=0}^\infty L^{k+1}P_S\alpha R^{-(k+1)}-\displaystyle\sum_{k=1}^\infty L^{(-k-1)}P_U\alpha R^{k-1}\right)\\
& = & \left(
\displaystyle\sum_{k=1}^\infty L^kP_S\alpha R^{-k}-\displaystyle\sum_{k=0}^\infty L^{k+1}P_S\alpha R^{-(k+1)}\right)\\
& & 
-\left(
\displaystyle\sum_{k=1}^\infty L^{-k}P_U\alpha R^k+P_U\alpha-\displaystyle\sum_{k=1}^\infty L^{-(k-1)}P_U\alpha R^{k-1}\right)\\
& & + P_S\alpha+P_U\alpha\\
& = & \alpha,
\end{eqnarray*}
$\Gamma$ is a right inverse of $\Phi$.

We shall prove that $\Gamma$ is the inverse of $\Phi$ restricted to a certain subspace of $C_b(X)$:
Clearly $S\cap L^{-1}(U)\subset S\cap U=\{0\}$ so the sum
$$
Y=S+L^{-1}(U)
$$
is direct.
Since the sum $S+U$ is direct and closed (for it is $X$), Theorem 1.1 in \cite{l} supplies $A>0$ such that  $\|x\|\leq A\|x-y\|$
for all $(x,y)\in S\times U$ and so for all $(x,y)\in S\times L^{-1}(U)$ too.
On the other hand, $S\cap L^{-1}(U)\subset S\cap U=\{0\}$ so the sum
$S+L^{-1}(U)$ is direct. Therefore, $Y$ is closed by Theorem 1.1 in \cite{l}.

It follows that $Y$ is a Banach space with the induced norm and so
$C_b(X,Y)=\{\phi\in C_b(X):\phi(X)\subset Y\}$ is a closed subspace of $C_b(X)$.
Now we need the following lemmas.

\begin{lem}
\label{l1-4}
$\Gamma=(\Phi|_{C_b(X,Y)})^{-1}$.
\end{lem}

\begin{proof}
Since $\Gamma$ is a right inverse of $\Phi$, it
suffices to show that
if $\Phi(\phi)=\alpha$ for $\phi\in C_b(X,Y)$ and $\alpha\in C_b(X)$, then
$\phi=\Gamma(\alpha)$.

Suppose that $\Phi(\phi)=\alpha$ i.e. $\phi R-L\phi=\alpha$. Then, $\phi R=L\phi +\alpha$ and so
\begin{equation}
\label{herma}
\phi R^n=L^n\phi+\displaystyle\sum_{k=1}^nL^{n-k}\alpha R^{k-1},
\quad\quad\forall n\in\mathbb{N}.
\end{equation}
Composing with $L^{-n}$ to the left above identity we get
$$
\phi=L^{-n} \phi R^n-\displaystyle\sum_{k=1}^nL^{-k}\alpha R^{k-1},
\quad\quad\forall n\in \mathbb{N}.
$$
We split this identity as
$\phi=y_n+z_n$ with
$$
y_n=L^{-n} P_S\phi R^n-\displaystyle\sum_{k=1}^nL^{-k}P_S\alpha R^{k-1}
\mbox{ and }
z_n=L^{-n} P_U\phi R^n-\displaystyle\sum_{k=1}^nL^{-k}P_U\alpha R^{k-1}
$$
We clearly have that $z_n\in L^{-1}(U)$ (i.e. $z_n(x)\in L^{-1}(U)$ for all $x\in X$)
for all $n\in \mathbb{N}$.

We claim that $y_n\in S$ for all $n\in\mathbb{N}$.
For $n=1$, $Ly_1=P_S\phi R-P_S\alpha\in S$.
Then, if $y_1=a_1+b_1\in S+U$, $Lb_1=Ly_1-La_1\in S$ so $b_1\in L^{-1}(S)$.
But also
$\phi=a_1+b_1+z_1$ and $z_1\in L^{-1}(U)$ so
$b_1=\phi-a_1-z_1\in S+L^{-1}(U)$ thus $b_1\in L^{-1}(U)$.
Summarizing, $b_1\in L^{-1}(S)\cap L^{-1}(U)=\{0\}$ so $b_1=0$ thus $y_1=a_1\in S$.
Now suppose that $y_n\in S$ for some $n\in\mathbb{N}$.
Since
\begin{eqnarray*}
Ly_{n+1} & = & L^{-n}P_S\phi R^{n+1}-\displaystyle\sum_{k=1}^{n+1}L^{-(k-1)}P_S\alpha R^{k-1}\\
& = & \left(L^{-n}P_S\phi R^n-\displaystyle\sum_{k=2}^{n+1} L^{-(k-1)}P_S\alpha R^{k-2}\right)R+P_S\alpha\\
& = &  \left(L^{-n}P_S\phi R^n-\displaystyle\sum_{k=1}^{n} L^{-k}P_S\alpha R^{k-1}\right)R + P_S\alpha\\
& = & y_nR+P_S\alpha,
\end{eqnarray*}
one has $Ly_{n+1}\in S$.
Then, if $y_{n+1}=a_{n+1}+b_{n+1}\in S+U$, $Lb_{n+1}=Ly_{n+1}-La_{n+1}\in S$ so $b_{n+1}\in L^{-1}(S)$.
But also
$\phi=a_{n+1}+b_{n+1}+z_{n+1}$ and $z_{n+1}\in L^{-1}(U)$ so
$b_{n+1}=\phi-a_{n+1}-z_{n+1}\in S+L^{-1}(U)$ thus $b_{n+1}\in L^{-1}(U)$.
Summarizing, $b_{n+1}\in L^{-1}(S)\cap L^{-1}(U)=\{0\}$ so $b_{n+1}=0$ thus $y_{n+1}=a_{n+1}\in S$.
Then, the claim is proved by induction.

It follows that
$$
P_U\phi=L^{-n}P_U\phi R^n-\displaystyle\sum_{k=1}^nL^{-k}P_U\alpha R^{k-1}.
$$
We have $\|L^{-n}P_U\phi R^n(x)\|\leq \|L^{-n}|_U\|\|\phi\|_\infty\to0$ as $n\to\infty$ hence
\begin{equation}
\label{peu1}
P_U\phi=-\displaystyle\sum_{k=0}^\infty L^{-k}P_U\alpha R^{k-1}.
\end{equation}
Likewise,
\begin{equation}
\label{peu2}
P_S\phi=\displaystyle\sum_{k=1}^\infty L^kP_S\alpha R^{-k-1}.
\end{equation}
Indeed, composing \eqref{herma} with $R^n$ to the right we get
$$
\phi=L^n\phi R^{-n}+\displaystyle\sum_{k=0}^{n-1}L^k\alpha R^{-k-1}
$$
As before we write
$\phi=y_n'+z_n'$ where
$$
y_n'=L^nP_S\phi R^{-n}+\displaystyle\sum_{k=0}^{n-1}L^kP_S\alpha R^{-k-1}
\mbox{ and }
z_n'=L^nP_U\phi R^{-n} +\displaystyle\sum_{k=0}^{n-1}L^k P_U \alpha R^{-k-1}
$$
This time we easily have $y_n'\in S$. 

We claim that $z_n'\in L^{-1}(U)$.
Indeed,
for $n=1$ we see that
$z'_1=LP_U\phi R^{-1}+P_U\phi R^{-1}$.
As $\phi\in Y=S+L^{-1}U$, $LP_U\phi R^{-1}\in U$.
As $P_U\alpha R^{-1}\in U$, $z_1'\in U$.
By $\phi=y_1'+z_1'$, $y'_1\in S$ and $\phi\in S+L^{-1}U$, we get $z'_1\in S+L^{-1}U$.
So, $z'_1\in L^{-1}U$ proving the result for $n=1$.
Now suppose the claim is true for $n$.
For $n+1$ we have
$$
z_{n+1}'=L^{n+1}P_U\phi R^{-(n+1)}+\displaystyle\sum_{k=0}L^kP_U\alpha R^{-k-1}=L z'_n R^{-1}+P_U\alpha R^{-1}\in U.
$$
As $\phi=y_{n+1}'+z_{n+1}'$, $y_{n+1}'\in S$ and $\phi\in S+L^{-1}U$, $z_{n+1}'\in S+L^{-1}U$.
So, $z'_{n+1}\in L^{-1}U$. This proves the claim by induction.

It follows from the claim that
$$
P_S\phi=L^nP_S\phi R^{-n}+\displaystyle\sum_{k=0}^{n-1}L^kP_S\alpha R^{-k-1},
\quad\quad\forall n\geq1.
$$
Since $\|L|_S\|\to0$ as $n\to\infty$, we obtain \eqref{peu2}.
Combining \eqref{peu1} and \eqref{peu2} we get
$$
\phi=P_S\phi+P_U\phi=\displaystyle\sum_{k=0}^\infty L^{-k}P_U\alpha R^{k-1}-\displaystyle\sum_{k=1}^\infty L^kP_S\alpha R^{-k-1}=\Gamma(\alpha).
$$
This completes the proof.
\qed
\end{proof}

We also need the following lemma.

\begin{lem}
\label{l2-4}
There is a unique $h\in C_b(X)$ such that the uniformly continuous map $H=I+h:X\to X$
satisfies
$$
HL=MH\quad\quad\mbox{ and }\quad\quad \|H-T\|_\infty\leq \epsilon.
$$
\end{lem}

\begin{proof}
Given $h\in C_b(X)$ we have that
$$
HL=MH\Longleftrightarrow (I+h)L=(L+\beta)(I+h)\Longleftrightarrow hL=Lh+\beta(I+h)
$$
$$
\Longleftrightarrow hL-Lh=\beta(I+h)\Longleftrightarrow \Phi_L(h)=\beta(I+h)\Longleftrightarrow h=\Phi_L^{-1}(\beta(I+h))
$$
where $\Phi_L$ is the map defined in \eqref{perrus} for $R=L$.
Then, $HL=MH$ if and only if $h$ is a fixed point of the
map $\Pi_1:C_b(X)\to C_b(X)$ defined by
$$
\Pi_1(\phi)=\Phi^{-1}_L(\beta(I+\phi)).
$$
By Lemma \ref{l1-4} one has
$$
Lip(\Pi_1)\leq \|\Gamma\| Lip(\beta)\leq d(A+B)\delta=\epsilon<1.
$$
Then, $\Pi_1$ has a unique fixed point and so there is a unique $h\in C_b(X)$
such that $H=I+h$ satisfies $HL=MH$.
Since
$$
\|h\|_\infty=\|\Phi_L^{-1}(\beta(I+h))\|_\infty\leq d(A+B)\|\beta\|_\infty
\leq d(A+B)\delta=\epsilon
$$
we are done.
\qed
\end{proof}

\begin{lem}
\label{l3-4}
There is a unique $h'\in C_b(X)$ such that
the uniformly continuous map $H'=I+h':X\to X$ satisfies
$$
H'M=LH'\quad\quad\mbox{ and }\quad\quad \|H'-I\|_\infty=\|h'\|_\infty\leq\epsilon.
$$
\end{lem}

\begin{proof}
We apply Lemma \ref{l1-4} to $R=M$ to get
the map $\Phi_M$.
We have
$$
H'M=LH'\Longleftrightarrow (I+h')M=L(I+h')\Longleftrightarrow M+h'M=L+Lh'\Longleftrightarrow
$$
$$
h'M-Lh'=L-M
\Longleftrightarrow h'M-Lh'=-\beta\Longleftrightarrow \Phi_M(h')=-\beta
\Longleftrightarrow h'=\Phi^{-1}_M(-\beta).
$$
This proves the existence and unicity of $h'$.
Finally since
$$
\|h'\|_\infty\leq \|\Phi^{-1}_M\|_\infty \|\beta\|_\infty\leq d(A+B)\delta=\epsilon
$$
we are done.
\qed
\end{proof}

To finish with the proof of the theorem we will prove
$$
H'H=HH'=I.
$$
By the previous two lemmas we have
$$
H'HL=H'MH=LH'H.
$$
As $H'H=I+h''$ with $h''=h'(I+h)\in C_b(X)$ we have $h''=0$ by unicity hence
$$
H'H=I.
$$
On the other hand, we also have
$$
HH'M=MHH'.
$$
As before $HH'=I+v$ with $v\in C_b(X)$ so
$(I+v)M=M(I+v)$.
Now,
$$
(I+v)M=M(I+v)\Longleftrightarrow vM-Lv=\beta(I+v)\Longleftrightarrow \Phi_M(v)=\beta(I+v)-\beta
$$
$$
\Longleftrightarrow v=\Phi_M^{-1}(\beta(I+v)-\beta)
$$
hence $v$ is a fixed point of the map
$\hat{\Phi}:C_b(X)\to C_b(X)$ defined by
$$
\hat{\Phi}(\phi)=\Phi^{-1}_M(\beta(I+\phi)-\beta).
$$
Since
$$
Lip(\hat{\Phi})\leq \|\Phi_M^{-1}\| Lip(\beta)\leq d(A+B)\delta=\epsilon<1,
$$
$\hat{\Phi}$ has a unique point.
But $\hat{\Phi}(0)=\Phi^{-1}_M(0)=0$ so $v=0$ thus
$$
HH'=I.
$$
This completes the proof.
\qed
\end{proof}

A consequence of Bernardes-Messaoudi Theorem \ref{thA-3} above is given by the generalized Grobman-Hartman Theorem below.

Let $F:X\to X$ be a $C^1$-map of a Banach space.
We say that a fixed point $p$ of $F$ is a {\em generalized hyperbolic fixed point}
if its derivative $DF_p:X\to X$ is a generalized hyperbolic linear homeomorphism of $X$.
For such fixed points we have the following result.

\begin{cor}[Generalized Grobman-Hartman Theorem]
\label{hartman}
For every $C^1$ diffeomorphism of a Banach space $F:X\to X$ and every generalized hyperbolic fixed point $p$ of $F$ there are a homeomorphism
$H:X\to X$ and a neighborhood $U$ of $p$ in $X$ such that
$H\circ F=DF_p\circ H$ on $U$.
\end{cor}

\begin{proof}
First we reduce the proof to the case $p=0$.
Consider the diffeomorphism $G(y)=F(y+p)-p$ for $y\in X$.
Then, $G(0)=0$ and $DG_0=DF_p$ so $0$ is a generalized hyperbolic fixed point of $G$.
If the result is valid for $p=0$, there is a homeomorphism
$K:X\to X$ and an open neighborhood $V$ of $0$ in $X$ such that
$$
K\circ G=DG_0\circ K \quad\mbox{ on }\quad V.
$$
Take $U=p+V$ and $H:X\to X$ defined by $H(x)=K(x-p)$ for $x\in X$.
If $x\in U$, $y=x-p\in V$ and so
\begin{eqnarray*}
H(F(x)) & = & K(F(x)-p)\\
& = & K(F(y+p)-p)\\
& = & K(G(y))\\
& = & DG_0(K(y))\\
& = & DF_0(H(x))
\end{eqnarray*}
proving $H\circ F=DF_p\circ H$ in $U$.
It remains to prove the result when $p=0$.

Put $L=DF_0$.
It follows from Theorem \ref{thA-3} that $L$ is strong structural stable.
Take $\delta>0$ from this property for $\epsilon=1$ and $\alpha=F-L$.
We have $\alpha(0)=F(0)=0$ and $D\alpha_0=0$.
Then, there is a neighborhood $U$ of $0$ such that
$\|\alpha|_U\|_\infty<\delta$ and $Lip(\alpha|_U)<\delta$.
From this we can extend $\alpha|_U$ to a bounded Lipschitz map $\beta:X\to X$ such that
$\|\beta\|_\infty\leq \delta$ and $Lip(\beta)\leq \delta$.
Then, there is a homeomorphism $H:X\to X$ such that
$H\circ (L+\beta)=L\circ H$.
So,
\begin{eqnarray*}
DF_0(H(x)) & = & L(H(x))\\
& = & H(L(x)+\beta(x))\\
& = & H(L(x)+F(x)-L(x)) \\
& = & H(F(x)),
\quad\quad\forall x\in U
\end{eqnarray*}
proving the result.
\end{proof}

\begin{rem}
It was proved also in [BMpams] that the topological conjugacy $H$ in the generalized Grobman-Hartman theorem above can be chosen to be Holder continuous.
\end{rem}

\section*{Remark about the strong structural stability}

\noindent
We start with the following result.

\begin{theo}
\label{car}
The following properties are equivalent for every $L\in GL(X)$:
\begin{enumerate}
\item
$r(L)<1$;
\item
there is $0<\gamma<\infty$ such that $\displaystyle\sum_{k=0}^\infty L^k(x_k)$ converges
and
$$
\left\| \displaystyle\sum_{k=0}^\infty L^k(x_k)\right\|\gamma \sup_{k\geq0}\|x_k\|,
$$
for every bounded sequence $(x_k)_{k\geq0}$ of $X$.
\end{enumerate}
\end{theo}

\begin{proof}
First suppose Item (1) i.e. $r(L)<1$. As before this is equivalent to
$$
\gamma=\displaystyle\sum_{k=0}^\infty\|L^k\|<\infty.
$$
Now if $(x_k)_{k\geq0}$ bounded,
$$
\left\|
\displaystyle\sum_{k=0}^\infty L^k(x_k)\right\|
\leq \displaystyle\sum_{k=0}^\infty\|L^k(x_k)\|\leq \displaystyle\sum_{k=0}^\infty\|L^k\|\|x_k\|\leq \gamma\sup_{k\geq0}\|x_k\|
$$
proving Item (2).

Conversely, suppose that Item (2) holds.
We shall prove that $L$ has the shadowing property.
Indeed, we show that $\gamma$ in the statement of Item (2) is a shadowableness constant.
Take a bounded sequence $(z_n)_{n\in\mathbb{Z}}$ and define
the sequence $(y_n)_{n\in\mathbb{Z}}$ by
$$
y_n=\displaystyle\sum_{k=0}^\infty L^k(z_{n-k-1}),\quad\quad\forall n\in\mathbb{Z}.
$$
It follows from Item (2) that $y_n$ is well defined. Moreover,
$$
\|y_n\|\leq \left\|\displaystyle\sum_{k=0}^\infty L^k(z_{n-k-1})\right\|
\leq \gamma\sup_{n\in\mathbb{Z}}\|z_n\|,\quad\quad\forall n\in\mathbb{Z}
$$
then
$$
\sup_{n\in\mathbb{Z}}\|y_n\|\leq \gamma\sup_{n\in\mathbb{Z}}\|z_n\|.
$$
Also,
\begin{eqnarray*}
y_{n+1} & = &\displaystyle\sum_{k=0}^\infty L^k(z_{n-k})\\
& = & \left( \displaystyle\sum_{k=1}^\infty L^k(z_{n-k})\right)+z_n\\
& = & \displaystyle\sum_{k=0}^\infty L^{k+1}(z_{n-k-1})+z_n\\
& = & L\left(\displaystyle\sum_{k=0}^\infty L^k(z_{n-k-1})\right)+z_n\\
& = & L(y_n)+z_n,\quad\quad\forall n\in\mathbb{Z},
\end{eqnarray*}
proving the assertion about $\gamma$.
Therefore, $L$ has the shadowing property.

Next we observe that $\displaystyle\sum_{k=0}^\infty L^k$ converges strongly and so
$L^k\to0$ strongly as $k\to\infty$.
Then, $\sup_{k\geq0}\|L^k\|<\infty$ by the Banach-Steinhouse theorem so
$r(L)\leq1$.
It follows that the spectrum $\sigma(L)$ of $L$ is contained in the closed unit disk
$D\subset \mathcal{C}$.
Now suppose that $\sigma(L)\cap \partial D\neq\emptyset$.
Since $\sigma\subset D$, $\sigma(L)\cap \partial D\subset \partial \sigma(L)$ and so
there is $\lambda\in \sigma_a(L)\cap \partial D$ where $\sigma_a(L)$ is the approximated point spectrum of $L$ (e.g. [Dowson]).
Now take $\delta>0$ from the shadowing property of $L$ for $\epsilon=\frac{1}2$.
Since $\lambda\in\sigma_a(L)$, there is an unitary vector $x\in X$ such that
$\|L(x)-\lambda x\|<\delta$.
It follows that
$$
\|L(\lambda^nx)-\lambda^{n+1}x\|=|\lambda^k|\|L(x)-\lambda x\|<\delta,\quad\quad\forall n\in\mathbb{Z},
$$
so there is $y\in X$ such that
$\|L^k(y)-\lambda^kx\|\leq\dfrac{1}2$ for $k\geq0$.
Then,
$$
\|L^k(y)\|\geq \|\lambda^k x\|-\|L^k(y)-\lambda^kx\|\geq 1-\dfrac{1}2=\dfrac{1}2,\quad\quad\forall k\geq0.
$$
Since $L^k\to0$ strongly as $k\to0$, we get a contradiction which proves the result.
\qed
\end{proof}

The above corollary looks to be true for noninvertible linear operators of a Banach space $L:X\to X$.
This can be used to prove the following corollary.

\begin{cor}
A necessary and sufficient condition for $L\in GL(X)$ to be generalized hyperbolic is that there are
a splitting $X=S\oplus U$ with $S$ and $U$ closed subspaces and $0<\gamma<\infty$ such that
$L(S)\subset S$, $L^{-1}(U)\subset U$ and
$\displaystyle\sum_{k=0}^\infty L^k(x_k)$ (resp. $\displaystyle\sum_{k=0}^\infty L^{-k}(x_k)$ )
converges and
$$
\left\|
\displaystyle\sum_{k=0}^\infty L^k(x_k)\right\|
\leq \gamma\sup_{k\geq0}\|x_k\|
\quad\quad\left(\mbox{ resp. }
\left\|
\displaystyle\sum_{k=0}^\infty L^{-k}(x_k)\right\|
\leq \gamma\sup_{k\geq0}\|x_k\|\right)
$$
for every bounded sequence $(x_k)_{k\geq0}$ in $S$ (resp. $U$).
\end{cor}

\chapter{Homoclinic points}

\noindent
Given a Banach space $X$ and $L\in GL(X)$ we say that $x\in X$ is a {\em homoclinic point}
if
$L^n(x)\to\infty$ as $n\to\pm\infty$.
We denote by
$H=H(L)$ the set of homoclinic points of $L$.

We have the following proposition
\begin{prop}
\label{mandeta}
If $L\in GL(X)$, then
$H$ is an invariant subspace which is closed if and only if $H=\{0\}$.
\end{prop}

\begin{proof}
Clearly $H$ is an invariant subspace.
If $H=\{0\}$, then $H$ is closed. Conversely, suppose that $H$ is closed.
Then, $H$ is a Banach space with the induced norm.
By definition we have
$L^n(x)\to\infty$ as $n\to\infty$ for every $x\in H$ and so $\sup_{n\in\mathbb{Z}}\|L^n(x)\|<\infty$ for every $x\in H$. Then,
$\sup_{n\in\mathbb{Z}}\|L^n\|<\infty$ by the Banach-Steinhouse theorem.
From this we obtain that $L|_H\in GL(H)$ is equicontinuous.

Now, take $x\in H$ and $\delta$ from the equicontinuity of $L|_H$ for $\epsilon=1$.
Since $x\in H$, there is $n_0\in\mathbb{N}$ such that
$\|L^{n_0}(x)\|<\delta$ and so $\|L^{n+n_0}(x)\|<1$ for every $n\in\mathbb{Z}$.
Taking $n=-n_0$ we get $\|x\|<1$. Since $x\in H$ is arbitrary, $H=\{0\}$.
\qed
\end{proof}

Now we study the homoclinic points of a generalized hyperbolic linear homeomorphism $L\in GL(X)$.
First we show the following result.

\begin{theo}
\label{pingos}
If $L\in GL(X)$ is generalized hyperbolic, then $L$ is hyperbolic if and only if $H=\{0\}$.
\end{theo}

\begin{proof}
If $L$ is hyperbolic, then $L$ is expansive and so $H\subset E^c=\{0\}$ hence $H=\{0\}$.
Conversely,
suppose that $H=\{0\}$ but $L$ is not hyperbolic.
Let $X=S\oplus U$ the splitting of $L$.
If $L(S)\supset S$ and $L^{-1}(U)\supset U$, then $L$ is hyperbolic by definition.
Hence we can assume by contradiction that $L^{-1}(U)\supset U$ (say).
Then, there is $x\in U$ such that $P_SL(x)\neq0$ where
$P_S$ and $P_U$ are the projections associated to the splitting $X=S\oplus U$.
Write
$L(x)=P_SL(x)+P_UL(x)$
so $x-L^{-1}P_SL(x)=L^{-1}P_UL(x)$. As $x,L^{-1}P_UL(x)\in U$, $L^{-1}P_SL(x)\in U$
and so
$L^{-n}P_SL(x)\to0$ as $n\to\infty$. But we also have $P_SL(x)\in S$ hence $L^nP_SL(x)\to0$ as $n\to\infty$ proving $P_SL(x)\in H$.
Since $H=\{)\}$, this contradicts $P_SL(x)\neq0$ completing the proof.
\qed
\end{proof}

Question \ref{q1} and the above theorem motivate the following question:

\begin{question}
Prove (or disprove) that if $X$ is a Banach space and $L\in GL(X)$ has the shadowing property,
then $L$ is hyperbolic if and only if $H=\{0\}$.
\end{question}

We have seem the role played by the subspace $L(U)\cap S$ in the proof of the above theorem (point $x$ found there is in that subspace).
Now we use it to obtain a dense subspace of $H$.
More precisely, the following result hold (see \cite{jlm}).

\begin{theo}
\label{ferida}
Suppose that $L\in GL(X)$
exhibits a splitting $X=A\oplus B$ formed by closed subspaces $A$ and $B$ with the following properties:
\begin{enumerate}
\item
$L(A)\subset A$ and $L^{-1}(B)\subset B$;
\item
$L^n|_A\to 0$ and $L^{-n}|_B\to0$ strongly as $n\to\infty$.
\end{enumerate}
Then,
$\displaystyle\bigcup_{n,m\geq0}(L^n(B)\cap L^{-m}(A))$ is an invariant dense subspace
of $H$.
\end{theo}

\begin{proof}
To simplify write
$$
\hat{H}=\displaystyle\bigcup_{n,m\geq0}(L^n(B)\cap L^{-m}(A)).
$$
If $x\in L^n(B)\cap L^{-m}(A)$,
$L^r(x)\to0$ as $r\to\pm\infty$ so $x\in H$.
Consequently, $L^n(B)\cap L^{-m}(A)\subset H$ for all $n,m\geq0$ hence
$$
\hat{H}\subset H.
$$
On the other hand, if $x_1\in L^{n_1}(B)\cap L^{-m_1}(A)$ and $x_2\in L^{n_2}(B)\cap L^{-m_2}(A)$ with $n_i,m_i\geq 0$ ($i=1,2$)
and $\lambda\in \mathbb{C}$, then
$$
x_1+x_2\in L^{\max\{n_1,n_2\}}(B)\cap L^{-\max\{m_1,m_2\}}(A)
\quad\mbox{ and }\quad
\lambda x_1\in L^{n_1}(B)\cap L^{-m_2}(A),
$$
so $\hat{H}$ is a subspace.
Since
$$
L(\hat{H})=\displaystyle\bigcup_{n,m\geq0}(L^{n+1}(B)\cap L^{-m}(L(A)))\subset \displaystyle\bigcup_{n\geq1,m\geq0}(L^n(B)\cap L^{-m}(A))
\subset \hat{H}
$$
and
$$
L^{-1}(\hat{H})=\displaystyle\bigcup_{n,m\geq0}(L^{n}(L^{-1}(B))\cap L^{-m-1}(A))\subset \displaystyle\bigcup_{n\geq0,m\geq1}(L^n(B)\cap L^{-m}(A))
\subset \hat{H},
$$
$L(\hat{H})= \hat{H}$. Therefore, $\hat{H}$ is invariant.

It remains to prove that $\hat{H}$ is dense in $H$.
Let $P_A:X\to A$ and $P_B:X\to B$ be the associated projections.
Take $x\in H$.
For all $n\in \mathbb{N}$ one has,
$$
L^{-n}P_A(x)=P_AL^{-n}P_A(x)+P_BL^{-n}P_A(x)
$$
so
$$
P_A(x)=L^nP_AL^{-n}P_A(x)+L^nP_BL^{-n}P_A(x).
$$
Since $P_A(x),L^nP_AL^{-n}P_A\in A$, we have
$$
L^nP_BL^{-n}P_A(x)\in A.
$$
But also
$P_BL^{-n}P_A(x)\in B$ so
$$
L^nP_BL^{-n}P_A(x)\in L^n(B)
$$
thus
\begin{equation}
\label{pec}
L^nP_BL^{-n}P_A(x)\in L^n(B)\cap S,\quad\quad\forall n\geq0.
\end{equation}
Moreover,
$x=P_A(x)+P_B(x)$ so $L^{-n}x=L^{-n}P_A(x)+L^{-n}P_B(x)$. Since $x\in H$ and $L^{-n}|_B\to0$ strongly, we get
$$
\lim_{n\to\infty} L^{-n}P_A(x)=\lim_{n\to0}(L^{-n}x-L^{-n}P_B(x))=0.
$$
Now we observe that since $L^n|_A\to 0$ strongly, $\sup_{n\geq0}\|L^n|_A\|<\infty$ by the Banach-Steinhouse theorem so
\begin{eqnarray*}
\lim_{n\to\infty}\|P_A(x)-L^nP_BL^{-n}P_A(x)\| & = & \lim_{n\to\infty}\|L^nP_AL^{-n}P_A(x)\| \\
& \leq &
(\sup_{n\geq0}\|L^n|_A\|)\|P_A\|\lim_{n\to\infty}\|L^{-n}P_A(x)\| \\
& = & 0.
\end{eqnarray*}

Then, \eqref{pec} implies
\begin{equation}
\label{temer}
P_A(x)\in \overline{\left( \bigcup_{n\geq0}(L^n(B)\cap A)\right)}\subset
\overline{\hat{H}}.
\end{equation}

On the other hand, for all $m\in \mathbb{N}$ one has,
$$
L^mP_B(x)=P_AL^mP_B(x)+P_BL^mP_B(x)
$$
so
$$
P_B(x)=L^{-m}P_AL^mP_B(x)+L^{-m}P_BL^mP_B(x).
$$
Since $P_B(x),L^{-m}P_BL^mP_B\in B$, we have
$$
L^{-m}P_AL^mP_B(x)\in B.
$$
But also
$P_AL^mP_B(x)\in A$ so
$$
L^{-m}P_AL^mP_B(x)\in L^{-m}(A)
$$
thus
\begin{equation}
\label{pec'}
L^{-m}P_AL^mP_B(x)\in B\cap L^{-m}A,\quad\quad\forall m\geq0.
\end{equation}
Again
$x=P_A(x)+P_B(x)$ so $L^mx=L^mP_A(x)+L^mP_B(x)$.
Since $x\in H$ and $L^m|_A\to0$ strongly, we get
$$
\lim_{m\to\infty} L^mP_B(x)=\lim_{m\to0}(L^mx-L^mP_A(x))=0.
$$
Again we observe that $L^{-m}|_B\to 0$ strongly so $\sup_{n\geq0}\|L^{-m}|_B\|<\infty$ by the Banach-Steinhouse theorem thus
\begin{eqnarray*}
\lim_{m\to\infty}\|P_B(x)-L^{-m}P_AL^mP_B(x)\| & = & \lim_{m\to\infty}\|L^{-m}P_BL^mP_B\| \\
& \leq &
(\sup_{m\geq0}\|L^{-m}|_B\|)\|P_B\|\lim_{m\to\infty}\|L^mP_B(x)\| \\
& = & 0.
\end{eqnarray*}
Then, \eqref{pec'} implies
\begin{equation}
\label{temer'}
P_B(x)\in \overline{\left( \bigcup_{m\geq0}(B\cap L^{-m}(A))\right)}\subset
\overline{\hat{H}}.
\end{equation}
Finally, $\hat{H}$ is a subspace so $\overline{\hat{H}}$ also is and so $\overline{\hat{H}}+\overline{\hat{H}}=\overline{\hat{H}}$.
Then, \eqref{temer} and \eqref{temer'} imply
$$
x=P_A(x)+P_B(x)\in \overline{\hat{H}}+
\overline{\hat{H}}=\overline{\hat{H}}
$$
completing the proof.
\qed
\end{proof}

Clearly, a generalized hyperbolic splitting $X=S\oplus U$ satisfies the hypothesis of Theorem \ref{ferida}
with $A=S$ and $B=U$.
From this we obtain the following corollary.

\begin{cor}
\label{lavo1}
If $L\in GL(X)$ is generalized hyperbolic with splitting $X=S\oplus U$,
then
$\displaystyle\bigcup_{n,m\geq0}(L^n(U)\cap L^{-m}(S))$ is an invariant dense subspace of $H$.
\end{cor}

An interesting question is whether every linear homeomorphism under the conditions of Theorem \ref{ferida}
is generalized hyperbolic. However, the answer is negative by the following example.
\begin{exam}
\label{tucides}
There are a Banach space $X$ and $L\in GL(X)$ without the shadowing property but satisfying the hypothesis of Theorem \ref{ferida}.
\end{exam}

\begin{proof}
(Based on Example 2.1 in \cite{k}).
Choose any sequence $(\lambda_n)_{n\in\mathbb{N}}$ in the open unit complex disk such that
$$
0<\inf_{n\in\mathbb{N}}\|\lambda_n\|\leq \sup_{n\in\mathbb{N}}\|\lambda_n\|=1\quad\mbox{ and }\quad 1\notin\overline{\{\lambda_n:n\in\mathbb{N}\}}.
$$
Define $X=l^1$ i.e.
$$
X=\{ \xi=(\xi_k)_{k\in\mathbb{N}}\in \mathbb{C}^\mathbb{N}\mbox{ such that }\displaystyle\sum_{k=0}^\infty|\xi_k|<\infty\}.
$$ 
Define $L:X\to X$ by $L(\xi)=(L(\xi)_k)_{k\in\mathbb{N}}$ where
$$
L(\xi)_k =\lambda_k\xi_k,\quad\quad\forall k\in\mathbb{N}.
$$
We have $L\in GL(X)$ and $\L\|\leq1$.

We assert that $L^n\to 0$ strongly as $n\to\infty$.
Fix $\xi\in X$ and $\epsilon>0$.
Take $N\in\mathbb{N}$ such that
$$
\displaystyle\sum_{k=N+1}^\infty|\xi_k|<\dfrac{\epsilon}2.
$$
Since $|\lambda_i|<1$ for $1\leq i\leq N$, there is $n_0\in\mathbb{N}$ such that
$$
\displaystyle\sum_{k=1}^N|\lambda^n_k\xi_k|<\dfrac{\epsilon}2,\quad\quad\forall n\geq n_0.
$$
Then,
\begin{eqnarray*}
\|L^n(\xi)\| & \leq & \displaystyle\sum_{k=1}^N|\lambda_k^n\xi_k|+\displaystyle\sum_{k=N+1}^\infty|\lambda^n_k\xi_k|\\
& \leq & \dfrac{\epsilon}2+\displaystyle\sum_{k=N+1}^\infty|\xi_k|\\
& \leq & \dfrac{\epsilon}2+\dfrac{\epsilon}2\\
& = & \epsilon
\end{eqnarray*}
$\forall n\geq n_0$
proving the assertion.
It follows that $L$ satisfies the hypothesis of Theorem \ref{ferida} with $A=X$ and $B=\{0\}$.

It remains to prove that $L$ does not have the shadowing property.
Assume by contradiction that it does.
If there is $\lambda\in\sigma(L)$ with $|\lambda|=1$ we have $\lambda\in \partial\sigma(L)$ (for $\|L\|\leq 1$).
Since $\partial\sigma(L)\subset \sigma_a(L)$ (the approximated point spectrum of $L$, see [Dowson]),
$\lambda\in \sigma_a(L)$.
Now let $\delta>0$ be given by the shadowing property for $\epsilon=\dfrac{1}2$.
Since $\lambda\in\sigma_a(L)$, there is $\xi\in X$ unitary such that $\|L(\xi)-\lambda \xi\|\leq\delta$.
We have $\|L(\lambda^n \xi)-\lambda^{n+1}\xi\|=|\lambda^n|\|L(\xi)-\lambda \xi\|\leq\delta$ for every $n\in\mathbb{Z}$ hence there is $\eta\in A$ such that
$\|L^n(\eta)-\lambda^n\xi\|\leq \dfrac{1}2$ for all $n\in\mathbb{Z}$.
Then, $\|L^n(\eta)\|\geq \|\lambda^n\xi\|-\|L^n(\eta)-\lambda^n\xi\| \geq \dfrac{1}2$ for $n\geq0$ contradicting that $L^n\to0$ strongly as $n\to\infty$.
Therefore, $L$ does not have the shadowing property completing the proof.
\qed
\end{proof}

\begin{remark}
It follows in particular that the above example is not generalized hyperbolic (by Theorem \ref{thA-2}).
It is possible to modify it to get one with $A\neq \{0\}\neq B$
(to prove that this modification has not the shadowing property one uses Theorem \ref{agripino}).
\end{remark}

Let us use Theorem \ref{ferida} to prove the following result (see Item (1) of Theorem 2 in \cite{cgp}).

\begin{cor}
\label{lavo2}
Let $X$ is a Banach space and $L\in GL(X)$ be generalized hyperbolic with splitting $X=S\oplus U$.
Then, $L$ is hyperbolic if and only if $L(U)\cap S=\{0\}$.
\end{cor}

\begin{proof}
If $L$ is hyperbolic, $H=\{0\}$ by Theorem \ref{pingos} so
$L(U)\cap S=\{0\}$ by Corollary \ref{lavo1}.

To prove the converse we observe that
given $x\in U$, 
$L(x)=P_SL(x)+P_UL(x)$ so
$L^{-1}P_SL(x)=x-L^{-1}P_UL(x)\in U$ thus
$P_SL(U)\subset L(U)$ hence
$$
P_SL(U)\subset L(U)\cap S.
$$
Replacing $U$ and $L$ by $S$ and $L^{-1}$ in the above argument we get
$$
P_UL^{-1}(S)\subset U\cap L^{-1}(S).
$$
Then, if $L(U)\cap S=\{0\}$, $P_SL(U)=\{0\}$ so $L(U)\subset U$ thus
$L(U)=U$ likewise
$L(S)=S$ then $L$ is hyperbolic.
\qed
\end{proof}

Let us give an application based on the following lemma.

\begin{lem}
\label{perseguido}
Let $X$ a Banach space and $W,R\in GL(X)$ be such that there is a splitting
$X=S\oplus U$ formed by closed subspaces $S$ and $U$ with the following properties:
\begin{enumerate}
\item
$W(S)\subset S$, $W^{-1}(U)\subset U$, $R(S)\subset S$ and $R^{-1}(U)\subset U$.
\item
$\|W^{-1}|_U\|\|R^{-1}\|<1$, $\|R\| \|W|_S\|<1$.
\end{enumerate}
Then, $L=R\circ W$ is generalized hyperbolic.
If additionally, $R(U)\cap S\neq\{0\}$, then $L$ is not hyperbolic.
\end{lem}

\begin{proof}
Note
$L(S)=R(W(S))\subset R(S)\subset S$ and $L^{-1}(U)=W^{-1}(R^{-1}(U))\subset W^{-1}(U)\subset U$.
In addition $\|L|_S\|\leq \|R\|\|W|_S\|<1$ and $\|L^{-1}|_U\|=\|W^{-1}|_U\|\|R^{-1}\|<1$.
Therefore,  $X=S\oplus U$ is a generalized hyperbolic splitting of $L$ and so $L$ is generalized hyperbolic.
Now assume $R(U)\cap S\neq\{0\}$.
We have $L(U)\cap S\subset H$ by Theorem \ref{ferida} and
$L(U)\cap S=R(W(U))\cap S=R(U)\cap S\neq\emptyset$ so $H\neq\{0\}$ thus
$L$ is not hyperbolic by Theorem \ref{pingos}. This ends the proof.
\qed
\end{proof}

An application of the above results is given below.

\begin{exam}
\label{lockdown}
There is a Banach space $X$ and $L\in GL(X)$ such that $L$ is generalized hyperbolic but not hyperbolic.
\end{exam}

\begin{proof}
Let $X$ be as in Example \ref{tucides}.
Fix $0<\alpha<1$ and define $W\in GL(X)$ by
$W(\xi)_k=\lambda_k\xi_k$ for $k\in\mathbb{Z}$ where
$\lambda_k=\alpha$ (for $k\leq0$) or $\frac{1}\alpha$ (for $k>0$).
We have the splitting $X=S\oplus U$ by closed subspaces $S$ and $U$ defined by
$$
S=\{\xi\in l^1:\xi_k=0,\,\ k>0\}\mbox{ and }U=\{\xi\in l^1:\xi_k=0,\,\, k\leq0\}.
$$
It follows from the definitions that $\|W^{-1}|U\|=\|W|_S\|=\alpha$.
Now define $R\in GL(X)$ by
$R(\xi)_k=\xi_{k+1}$ for $k\in\mathbb{Z}$.
It follows that $R$ is an isometry so $\|R\|=\|R^{-1}\|=1$ thus
$\|W^{-1}|_U\|\|R^{-1}\|=\|R\|\|W|_S\|=\alpha<1$.
Therefore, $L=R\circ W$ is generalized hyperbolic by Lemma \ref{perseguido}.
Since $e_0\in R(U)\cap S$ (where $e_0\in X$ denotes the sequence with $1$ in the zero entry and $0$ in the other entries), we have $R(U)\cap L\neq\{0\}$ so $L$ is not hyperbolic.
\qed
\end{proof}

\chapter{A strange operator and applications}

\noindent
Let $X$ be a Banach space.
Define
$$
l^\infty(X)=\{\xi=\xi(\xi_n)_{n\in\mathbb{Z}}\in X^\mathbb{Z}:\sup_{n\in\mathbb{Z}}\|\xi_n\|<\infty\}
$$
namely the set of bounded sequences in $X$.
It follows that $l^\infty(X)$ is a Banach space if endowed
with the operations
$\xi+\eta=(\xi_n+\eta_n)_{n\in\mathbb{Z}}$, $\lambda\xi=(\lambda\xi_n)_{n\in\mathbb{Z}}$
for $\xi,\eta\in l^\infty(X)$ and $\lambda\in\mathbb{C}$
with the norm
$$
\|\xi\|=\sup_{n\in\mathbb{Z}}\|\xi_n\|.
$$
Given $L\in GL(X)$ we define the map
$L_\infty:l^\infty(X)\to l^\infty(X)$ by $L_\infty(\xi)=(L_\infty(\xi)_n)_{n\in\mathbb{Z}}$ where
$$
L_\infty(\xi)_n=\xi_{n+1}-L(\xi_n).
$$
We first note that $L_\infty$ is a linear map.
Since
$$
\|L_\infty(\xi)\|=\sup_{n\in\mathbb{Z}}\|\xi_{n+1}-L(\xi_n)\|\leq (1+\|L\|)\|\xi\|
$$
$\forall \xi\in l^\infty(X),$
$L_\infty:l^\infty(X)\to l^\infty(X)$ is a bounded operator with norm
$$
\|L_\infty\|\leq 1+\|L\|.
$$
We also have
\begin{equation}
\label{helado}
\|L_\infty-L'_\infty\|\leq \|\leq\|L-L'\|,\quad\quad\forall L,L'\in GL(X).
\end{equation}
Since
$$
L_\infty(\xi)=0 \iff L(\xi_n)=\xi_{n+1},\,\,\quad\forall n\in\mathbb{Z} \iff \xi=(L^n(\xi_0))_{n\in\mathbb{Z}}
$$
we obtain the following characterization
 $$
 Ker(L_\infty)=\{(L^n(x))_{n\in\mathbb{Z}}:x\in E^c\}.
$$
From this we deduce the following result.
\begin{prop}
$L$ is expansive if and only if
$L_\infty$ is injective.
\end{prop}

Next we use $L_\infty$ to compute $Shad(L)$.
First we prove the following lemma.

\begin{lem}
\label{4abril}
Let $X$ be a Banach space and $L\in GL(X)$.
\begin{enumerate}
\item
If $c>0$ and $B(0,c)\subset L_\infty(B(0,1))$,
then $(1+\epsilon)c^{-1}$ is a shadowableness constant of $L$ for every $\epsilon>0$.
\item
If $K$ is a shadowableness constant of $L$, then
$B(0,K^{-1})\subset L_\infty(B(0,1))$.
\end{enumerate}
\end{lem}

\begin{proof}
To prove Item (1) we
fix $c>0$ such that
$B(0,c)\subset L_\infty(B(0,1))$ and $\epsilon>0$. 
We shall prove that $K=(1+\epsilon)c^{-1}$ is a shadowableness constant of $L$.
Take a bounded sequence $\xi=(z_n)_{n\in\mathbb{Z}}$ in $X$. Then, $\xi\in l^\infty(X)$ and we can assume $\xi\neq0$.
It follows that 
$$
\frac{c}{(1+\epsilon)\|\xi\|}\xi\in B(0,c)
$$
and then
$$
L_\infty(\eta)=\frac{c}{(1+\epsilon)\|\xi\|}
$$
for some $\eta\in B(0,1)$.
So,
$$
L_\infty\left(\frac{(1+\epsilon)\|\xi\|}{c}\eta\right)=\xi
$$
thus the sequence $(y_n)_{n\in\mathbb{Z}}=\frac{(1+\epsilon)\|\xi\|}{c}\eta$
satisfies
$y_{n+1}=L(y_n)+z_n$ for all $n\in\mathbb{Z}$
and
$$
\sup_{n\in\mathbb{Z}}\|y_n\|=\left\|\frac{(1+\epsilon)\|\xi\|}{c}\eta\right\|
\leq \frac{1+\epsilon}c\|\xi\|=K\sup_{n\in\mathbb{Z}}\|z_n\|.
$$
It follows that $K$ is a shadowableness constant proving Item (1).

To prove Item (2) fix a shadowableness constant $K$ of $L$ (notice that $K>0$).
Take
$\xi\in l^\infty(X)$ with $\|\xi\|<K^{-1}$.
Then,
$\xi=(z_n)_{n\in\mathbb{Z}}$ is bounded and so
there is a sequence $\eta=(y_n)_{n\in\mathbb{Z}}$ such that
$$
\|\eta\|=\sup_{n\in\mathbb{Z}}\|y_n\|\leq K\sup_{n\in\mathbb{Z}}\|z_n\|=K\|\xi\|<1
\mbox{ and }
y_{n+1}=L(y_n)+\xi_n,\quad\forall n\in\mathbb{Z}.
$$
Then,
$\eta\in B(0,1)$ and the equality above implies
$L_\infty(\eta)=\xi$.
Therefore, $B(0,K^{-1})\subset L_\infty(B(0,1)$ proving the result.
\qed
\end{proof}

\begin{theo}
\label{carrot}
For every Banach space $X$ one has
$$
Shad(L)=\inf\{c^{-1}:c>0\mbox{ and } B(0,c)\subset L_\infty(B(0,1))\},
\quad\forall L\in GL(X).
$$
\end{theo}

\begin{proof}
Given $L\in GL(X)$ we have to prove $Shad(L)=I$, where
$$
I=\inf\{c^{-1}:c>0\mbox{ and } B(0,c)\subset L_\infty(B(0,1))\}.
$$
First we show
\begin{equation}
\label{pupa}
Shad(L)\leq I
\end{equation}
If $I=\infty$, then nothing to prove.
So, we can assume $I<\infty$.
Take $\epsilon>0$ and $c>(I+\epsilon)^{-1}$ such that
$B(0,c)\subset L_\infty(B(0,1))$.
It follows from Item (1) of Lemma \ref{4abril} that
$(1+\epsilon)c^{-1}$ is a shadowableness constant of $L$. Then,
$$
Shad(L)\leq (1+\epsilon)c^{-1}<(1+\epsilon)(I+\epsilon).
$$
Letting $\epsilon\to0$ we get
$Shad(L)\leq I$
proving \eqref{pupa}.

Next we prove
\begin{equation}
\label{papu}
I\leq Shad(L).
\end{equation}
Again if $Shad(L)=\infty$, then nothing to prove.
Take a shadowableness constant $K$ of $L$.
Then, $B(0,K^{-1})\subset L_\infty(B(0,1))$ by Item (2) of
Lemma \ref{4abril}
so $I\leq (K^{-1})^{-1}=K$ thus \eqref{papu} completing the proof.
\qed
\end{proof}

A direct corollary of the above theorem is given below.

\begin{cor}
\label{bangu}
If $X$ is a Banach space, then $L\in GL(X)$ has the shadowing property if and only if $L_\infty$ is onto.
\end{cor}

\begin{proof}
$L_\infty$ is onto $\iff$ $\exists c>0$ such that $L_\infty(B(0,c))\subset L_\infty(B(0,1))$
$\iff$ the infimum in Theorem \ref{carrot} is finite $\iff$ $Shad(L)$ is finite $\iff$ $L$ has the shadowing property.
\qed
\end{proof}

Another interesting corollary is as follows.

\begin{cor}
\label{al}
If $X$ is a Banach space, then
$Shad(L)\geq \|L_\infty\|^{-1}$ (hence $Shad(L)\neq0$) for all $L\in GL(X)$.
\end{cor}

\begin{proof}
If $c>0$ and $B(0,c)\subset L_\infty(B(0,1))$,
then
$$
\|L_\infty\|=\displaystyle\sup_{\|\xi\|\leq1}\|L_\infty(\xi)\|\geq \displaystyle\sup_{\eta\in B(0,c)}\|\eta\|=c
$$
so
$\|L_\infty\|^{-1}\leq c^{-1}$
thus Theorem \ref{carrot} applies.
\qed
\end{proof}

From this we get the following result from \cite{cgp}.

\begin{theo}
\label{ful}
The set of linear homeomorphism with the shadowing property of a Banach space is open with respect to the uniform topology.
\end{theo}

\begin{proof}
Let $X$ a Banach space and $L\in GL(X)$ be with the shadowing property.
Then, $L_\infty:l^\infty(X)\to l^\infty(X)$ is onto by Corollary \ref{bangu}
and so there is $\epsilon>0$ such that
every bounded linear operator $\psi':l^\infty(X)\to l^\infty(X)$ with
$\|l_\infty-\psi'\|<\epsilon$ is onto.
Now take $L'\in GL(X)$ with $\|L-L'\|<\epsilon$.
Then, \eqref{helado} implies
$\|L_\infty-L'_\infty\|<\epsilon$ so $L_\infty'$ is onto thus
$L'$ has the shadowing property by Corollary \ref{bangu}.
This finishes the proof.
\qed
\end{proof}

Recall that a map of a metric space $f:X\to X$ is {\em transitive} if there is $x\in X$ such that
the sequence $(f^n(x))_{n\in \mathbb{N}}$ is dense in $X$.
Transitive linear operators are usually called {\em hypercyclic operators}.
To contrast with the above theorem we will present the proof in \cite{bm1} of the following result.

\begin{theo}
\label{pupa}
The set of transitive linear homeomorphisms of a Banach space has empty interior with respect to the uniform topology.
\end{theo}

\begin{proof}
Let $X$ be a Banach space, $L\in GL(X)$ and $\epsilon$ be a positive number.
We shall find a nontransitive  $T\in GL(X)$ with $\|T-L\|<\epsilon$.
We can assume that every linear operator $T:X\to X$ with $\|T-L\|<\epsilon$ is a homeomorphism.

We know that the approximated point spectrum $\sigma_a(L^*)$ of the adjoint operator
$L^*:X^*\to X^*$ is never empty \cite{d}. So, we can choose $\lambda\in \sigma_a(L^*)$.
Then, there is $y^*\in X^*$ with $\|y^*\|=1$ such that
$\|L^*(y^*)-\lambda y^*\|<\frac{\epsilon}2$.
Choose $y\in X$ with $\|y\|=1$ such that
$|y^*(y)|>\frac{1}2$.
Define the map $R:X^*\to X^*$ by
$$
R(x^*)=\frac{x^*(y)}{y^*(y)}(\lambda y^*-L^*(y^*)),
\quad\quad x^*\in X^*.
$$
First we observe that $R$ is linear.
Also
$$
\|R(x^*)\|\leq 2\|\lambda y^*-L^*(y^*)\|\|y\|\|x^*\|
<\epsilon\|x^*\|
$$
proving that $R$ is bounded with $\|R\|<\epsilon$.
We also have that $R$ is w$^*$-to-$w^*$ continuous and then there is
a bounded linear operator $Q:X\to X$ such that $R=Q^*$ (see
Theorem 3.1.11 in \cite{m}).
It follows that $\|Q\|=\|Q^*\|=\|R\|<\epsilon$ and so
$T=L+Q\in GL(X)$.
Since
$$
T^*(y^*)=L^*(y^*)-Q^*(y^*)=L^*(y^*)+R(y^*)=L^*(y^*)+\lambda y^*-L(y^*)=\lambda y^*,
$$
we get $T^*(y^*)=\lambda y^*$.
Now suppose by contradiction that $T$ were transitive.
Then, there would exist $x\in X$ such that $(T^n(x))_{n\in \mathbb{N}}$ is dense in $X$.
Since $y^*\neq 0 $ (for $\|y^*\|=1$), $y^*$ is onto $\mathbb{C}$ and so
$(y^*(T^n(x)))_{n\in\mathbb{N}}$ would be dense in $\mathbb{C}$.
But
$y^*(T^n(x))=((T^*)^n(y^*))(x)=\lambda^n y^*(x)$ so the sequence $(y^*(T^n(x)))_{n\in\mathbb{N}}$ cannot be dense in $\mathbb{C}$. This is a contradiction which completes the proof.
\qed
\end{proof}

Summarizing, we have proved:
\begin{enumerate}
\item
Every $L\in GL(X)$ with the shadowing property is
{\em robustly shadowable} (i.e. every linear operator nearby $L$ with respect to the uniform topology
has the shadowing property).
\item
There are not
{\em robustly transitive linear homeomorphisms} (i.e. every linear homeomorphism can be approximated with respect to the uniform topology by a nontransitive operators).
\end{enumerate}
To complete the picture it is natural to ask if there are linear homeomorphisms which are {\em robustly expansive}
i.e. with the property that every nearby linear operator with respect to the uniform topology is expansive.
The similar notion of {\em robustly asymptotically expansive} can be considered too.
Indeed, such homeomorphisms exist e.g. the hyperbolic ones.
The question is if these are the sole ones namely,

\begin{ques}
Is every robustly expansive linear homeomorphism of a Banach space hyperbolic?
\end{ques}

\backmatter
\printindex


\end{document}